\documentclass{amsart}
\usepackage{hyperref} 
\usepackage{amsmath, amsthm, amssymb, amscd}
\usepackage{enumerate}
\usepackage{hyperref}
\usepackage{verbatim}
\usepackage{tikz}
\usepackage{tikz-cd}
\usepackage[noadjust]{cite}

\DeclareMathOperator{\diam}{diam}

\title{On the Lipschitz dimension of Cheeger--Kleiner}
\author{Guy C. David}
\address{Department of Mathematical Sciences\\ Ball State University, Muncie, IN 47306}
\email{gcdavid@bsu.edu}
\date{\today}


\begin{document}
\begin{abstract}
In a 2013 paper \cite{CK13_inverse}, Cheeger and Kleiner introduced a new type of dimension for metric spaces, the ``Lipschitz dimension''. We study the dimension-theoretic properties of Lipschitz dimension, including its behavior under Gromov-Hausdorff convergence, its (non-)invariance under various classes of mappings, and its relationship to the Nagata dimension and Cheeger's ``analytic dimension''. We compute the Lipschitz dimension of various natural spaces, including Carnot groups, snowflakes of Euclidean spaces, metric trees, and Sierpi\'nski carpets. As corollaries, we obtain a short proof of a quasi-isometric non-embedding result for Carnot groups and a necessary condition for the existence of non-degenerate Lipschitz maps between certain spaces.
\end{abstract}
\maketitle

\theoremstyle{plain}
\newtheorem{theorem}{Theorem}
\newtheorem{exercise}{Exercise}
\newtheorem{corollary}[theorem]{Corollary}
\newtheorem{scholium}[theorem]{Scholium}
\newtheorem{claim}[theorem]{Claim}
\newtheorem{observation}[theorem]{Observation}
\newtheorem{lemma}[theorem]{Lemma}
\newtheorem{sublemma}[theorem]{Lemma}
\newtheorem{proposition}[theorem]{Proposition}
\newtheorem{conjecture}{theorem}
\newtheorem{maintheorem}{Theorem}
\newtheorem{maincor}[maintheorem]{Corollary}
\renewcommand{\themaintheorem}{\Alph{maintheorem}}

\theoremstyle{definition}
\newtheorem{fact}[theorem]{Fact}
\newtheorem{example}[theorem]{Example}
\newtheorem{definition}[theorem]{Definition}
\newtheorem{remark}[theorem]{Remark}
\newtheorem{question}[theorem]{Question}

\numberwithin{equation}{section}
\numberwithin{theorem}{section}

\newcommand{\obar}[1]{\overline{#1}}
\newcommand{\haus}[1]{\mathcal{H}^n(#1)}
\newcommand{\prob}{\mathbb{P}}
\newcommand{\Tan}{\text{Tan}}
\newcommand{\WTan}{\text{WTan}}
\newcommand{\CTan}{\text{CTan}}
\newcommand{\CWTan}{\text{CWTan}}
\newcommand{\LIP}{\text{LIP}}
\newcommand{\dist}{\text{dist}}
\newcommand{\RR}{\mathbb{R}}
\newcommand{\HH}{\mathcal{H}}
\newcommand{\BB}{\mathcal{B}}
\newcommand{\bH}{\mathbb{H}}
\newcommand{\G}{\mathbb{G}}

\tableofcontents

\section{Introduction}

In a 2013 paper \cite{CK13_inverse}, Cheeger and Kleiner introduced a new type of dimension for metric spaces, the Lipschitz dimension, and proved some deep results about spaces of Lipschitz dimension $\leq 1$. In this paper, we study the dimension-theoretic properties of Lipschitz dimension. We begin our introduction with a discussion of the analogies with topological dimension that lead to the definition of Lipschitz dimension, and then describe the structure and results of the present paper.

\subsection{Topological dimension and Lipschitz dimension}\label{subsec:intro1}
We will be concerned with a metric analog of a well-studied concept in topology: the \textit{topological dimension} $\dim_T(X)$ of a space $X$. In the setting of compact metric spaces, the topological dimension $\dim_T(X)$ admits many equivalent definitions. The ``small inductive definition'' defines the empty set to have $\dim_T(\emptyset)=-1$, and then declares that $\dim_T(X)\leq n$ if $X$ has a neighborhood basis  of open sets $U$ with $\dim_T (\partial U) \leq n-1$. The ``Lebesgue covering definition'' declares $\dim_T(X)$ to be the minimal $n$ such that every locally finite open cover of $X$ admits a locally finite refinement of multiplicity at most $n+1$, meaning that every point is contained in at most $n+1$ sets of the refinement. 

These two definitions are known to be equivalent for compact (in fact, for separable) metric spaces (see \cite[Sections I.4 and II.5]{Nagata}), and so we refer to them simply as ``topological dimension'', denoted by $\dim_T$.

There is yet a another way (among many others unmentioned here) to view the topological dimension of a compact metric space $X$, this time through studying continuous maps from $X$ to Euclidean space. A continuous map is called \textit{light} if $f^{-1}(p)$ is totally disconnected for each $p$ in the image of $f$. We then have, for a compact metric space $X$, that
\begin{equation}\label{eq:toplight}
\dim_T(X) = \min\{ n\geq 0 : \exists f\colon X \rightarrow \RR^n \text{ light }\},
\end{equation}
where $\RR^0$ denotes the one-point space. (This follows from \cite[Theorems III.6 and III.10]{Nagata}.) Thus, the topological dimension of compact metric spaces can be seen through examining light maps to Euclidean space.

In \cite{CK13_inverse}, Cheeger and Kleiner were inspired by this fact to give a quantitative analog of topological dimension. They replace continuous maps by Lipschitz maps, give a quantitative analog of the notion of lightness, and then use the analog of \eqref{eq:toplight} to define a new notion of dimension.

As a preliminary, we need the following discrete notion:
\begin{definition}
For $r>0$, a finite sequence $(x_1, x_2, \dots, x_k)$ in a metric space $X$ is an \textit{$r$-path} if $d(x_i, x_{i+1})\leq r$ for all $i\in\{1,\dots,k\}$.

We say that two points in $X$ are in the same \textit{$r$-component} of $X$ if there is an $r$-path in $X$ containing both of them. This defines an equivalence relation on $X$.
\end{definition}

Cheeger and Kleiner then used this notion to define a quantitative analog of lightness for Lipschitz maps:

\begin{definition}[Cheeger--Kleiner \cite{CK13_inverse}]\label{def:LL2}
A map $f:X\rightarrow Y$ between metric spaces is \textit{Lipschitz light} if there is a constant $C>0$ such that 
\begin{itemize}
\item $f$ is Lipschitz with constant $C$, and
\item for every $r>0$ and every subset $W\subset Y$ with $\diam(W)\leq r$, the $r$-components of $f^{-1}(W)$ have diameter at most $Cr$.
\end{itemize}
\end{definition}

(An astute reader may note that Definition \ref{def:LL2} is not precisely the one given in \cite[Definition 1.14]{CK13_inverse}, though it is the one used in \cite[Section 11]{CK13_inverse}. We address this small discrepancy in subsection \ref{subsec:discrepancy} below.)

By analogy with \eqref{eq:toplight}, Cheeger and Kleiner define the following notion of dimension, which is the main subject of this paper.

\begin{definition}[Cheeger--Kleiner \cite{CK13_inverse}]\label{def:lipdim}
A metric space $X$ has \textit{Lipschitz dimension} $\leq n$ if there is a Lipschitz light map $f:X\rightarrow \mathbb{R}^n$.

We let the Lipschitz dimension of $X$ be the minimal $n$ such that $X$ has Lipschitz dimension $\leq n$, and denote this by $\dim_L(X)$. If $X$ admits no Lipschitz light map into any Euclidean space, we write $\dim_L(X) = \infty$.
\end{definition}
$\RR^0$ is again considered here to be the one-point metric space.

Because Lipschitz light maps are easily seen to be light maps, we immediately have the following consequence of \eqref{eq:toplight}, which we record for later use:

\begin{observation}\label{obs:liptop}
For a compact metric space $X$, $\dim_L (X) \geq \dim_T (X)$.
\end{observation}

\begin{remark}\label{rmk:liptop}
The characterization of topological dimension in \eqref{eq:toplight} serves for us as an analogy leading to Definition \ref{def:lipdim}, which we may then consider for non-compact metric spaces.

However, when leaving the realm of compact metric spaces, one has to be a bit careful with \eqref{eq:toplight} (or with the definition of ``light''), and hence with Observation \ref{obs:liptop} which relies on it. There is an example due to Erd\"os \cite{Erdos} of a totally disconnected separable metric space of topological dimension $1$. On the other hand, such a space trivially admits a light map to $\RR^0$, so \eqref{eq:toplight} fails for it.

On a positive note, both \eqref{eq:toplight} and Observation \ref{obs:liptop} extend immediately to $\sigma$-compact (in particular, proper) metric spaces by \cite[Theorem II.1]{Nagata}, which is the only situation in which we will apply these facts.

For a more general characterization of topological dimension through maps with ``$0$-dimensional fibers'' in a slightly stronger sense, see \cite[Theorem III.10]{Nagata}.
\end{remark}

After a brief aside to discuss the results of \cite{CK13_inverse} concerning Lipschitz dimension, we are ready to elaborate on the goals and results of the present paper.

\subsection{The results of Cheeger--Kleiner concerning Lipschitz dimension $1$}
The two main theorems of \cite{CK13_inverse} concern the structure of spaces of Lipschitz dimension $\leq 1$. We include these here by way of background; they are not used in the present paper.

First of all, Cheeger and Kleiner characterize spaces of Lipschitz dimension $\leq 1$ as limits of certain systems of metric graphs.
\begin{definition}[\cite{CK13_inverse}, Definition 1.8]
An inverse system 
$$ \dots \xleftarrow{\pi_{-i-1}} X_{-i}\xleftarrow{\pi_{-i}}\dots \xleftarrow{\pi_{1}} X_0 \xleftarrow{\pi_{0}} \dots \xleftarrow{\pi_{i-1}} X_i \xleftarrow{\pi_{i}} \dots $$
is \text{admissible} if, for some integer $m\geq 2$, the following conditions hold:
\begin{enumerate}[(i)]
\item $X_i$ is a non-empty directed graph for each $i\in \mathbb{Z}$.
\item For each $i\in\mathbb{Z}$, if $X'_i$ denotes the directed graph formed by subdividing each edge of $X_i$ into $m$ edges, then $\pi_i$ induces a map from $X_{i+1}$ to $X'_i$ that is simplicial, direction-preserving, and an isomorphism on each edge.
\item For every $i,j\in\mathbb{Z}$ and every $x\in X_i$, $x'\in X_j$, there is $k\leq \min(i,j)$ such that $x$ and $x'$ project into the same connected component of $X_k$.
\end{enumerate}
\end{definition}

Each graph $X_i$ is endowed with the path metric (allowing infinite distances between points in different connected components). The inverse limit $X_\infty$ admits a projection $\pi^\infty_i:X_\infty\rightarrow X_i$ for each $i\in\mathbb{Z}$. The space $X_\infty$ is then endowed with the metric $\overline{d}_\infty$: the supremal pseudo-distance on $X_\infty$ such that, for each $i\in\mathbb{Z}$ and $v\in X_i$, the inverse image $(\pi^\infty_i)^{-1}\left(\text{St}(v,X_i)\right)$ has diameter at most $2m^{-i}$. Here $\text{St}(v,X_i)$ denotes the closed star of $v$ in $X_i$.

\begin{theorem}[\cite{CK13_inverse}, Theorems 1.10 and 1.11]\label{theorem:CKgraph}
A compact metric space has Lipschitz dimension $\leq 1$ if and only if $X$ is bi-Lipschitz equivalent to an inverse limit of an admissible inverse system of graphs.
\end{theorem}
We refer the reader to \cite{CK13_inverse} for further details. In fact, Theorem \ref{theorem:CKgraph} admits a generalization to higher dimensions (and further): see \cite[Section 11]{CK13_inverse}.

Most of \cite{CK13_inverse} concerns the following embedding result.
\begin{theorem}[\cite{CK13_inverse}, Theorem 1.16]\label{theorem:CKL1}
Each compact metric space of Lipschitz dimension $\leq 1$ admits a bi-Lipschitz embedding into the Banach space $L_1(Z,\mu)$, for some measure space $(Z,\mu)$.
\end{theorem}

By contrast, \cite{La00} and \cite{CK13_PI} construct spaces of Lipschitz dimension $1$ with no bi-Lipschitz embedding into Hilbert space, or even the Banach space $\ell_1$. Furthermore, in \cite{LeeSid}, Lee and Sidiropoulos construct a space of Lipschitz dimension $2$ with no bi-Lipschitz embedding into $L_1$, so Theorem \ref{theorem:CKL1} does not extend to spaces of higher Lipschitz dimension.

\subsection{Purpose and results of the present paper}
The purpose of this paper is to study of the dimension-theoretic properties of Lipschitz dimension. We explain the structure of our paper and the ideas of our main results here, referring the reader to the appropriate sections for the official statements of theorems.

After giving basic notation in Section \ref{sec:notation}, we first address the behavior of Lipschitz dimension under products and unions in Section \ref{sec:productsunions}.

In Section \ref{sec:GH}, the technical core of the paper, we characterize Lipschitz light maps on doubling metric spaces via their behavior under Gromov-Hausdorff convergence (Theorem \ref{prop:LLGH}), obtaining bounds on Lipschitz dimension of tangent spaces as a consequence.

We then use this and other techniques to compute the Lipschitz dimension of a number of natural examples in Section \ref{sec:examples}, including metric trees, snowflakes of Euclidean spaces, and Carnot groups. The most concrete results of this section are that
\begin{itemize}
\item products of $n$ metric trees (as well as rank-$n$ Euclidean buildings) have Lipschitz dimension $n$ (Corollary \ref{cor:treebuilding}),
\item snowflakes of $\RR^n$ have Lipschitz dimension $n$ (Corollary \ref{cor:Rnsnowflake}), and
\item non-abelian Carnot groups have infinite Lipschitz dimension (Theorem \ref{thm:Carnot}).
\end{itemize}
As a corollary of this last fact, we obtain a short proof of a quasi-isometric non-embedding result (Corollary \ref{cor:carnotcoarse}) in the spirit of Pauls \cite{Pauls}. 

Also in Section \ref{sec:examples}, in Theorem \ref{theorem:ss}, we introduce a ``self-covering'' property for Euclidean subsets and use this to compute the Lipschitz dimension of some classical fractals, like the Sierpi\'nski carpets and gasket. The results of Section \ref{sec:examples} rely on Gromov-Hausdorff convergence arguments, along with some key ideas from \cite{LS} in the case of trees and buildings.

In Section \ref{sec:relation}, we briefly describe the relationship (or lack thereof) between Lipschitz dimension and other well-studied notions of metric dimension: the Nagata, Assouad, and Hausdorff dimensions. In particular, we show that Lipschitz dimension bounds Nagata dimension from above (Corollary \ref{cor:naglipdim}), and that the two agree for $0$-dimensional spaces but not in general (Proposition \ref{prop:zerodim}).

In Section \ref{sec:Cheeger}, we consider the ``Lipschitz differentiability spaces'' first described by Cheeger. These are metric measure spaces $X$ that carry a type of measurable cotangent bundle allowing for the almost-everywhere differentiation, in an appropriate sense, of Lipschitz functions from $X$ to $\RR$. Our main result in this section is Theorem \ref{thm:Cheegerbound}, which states that the dimension of Cheeger's cotangent bundle is bounded above by the Lipschitz dimension, complementing earlier results of Schioppa \cite[Corollary 5.99]{Schioppa} and the author \cite[Corollary 8.5]{GCD15} concerning Assouad dimension.

Lastly, in Section \ref{sec:mappingproperties}, we study the invariance and non-invariance properties of Lipschitz dimension under various categories of mappings: Lipschitz light, quasisymmetric, snowflake, and David--Semmes regular mappings. We provide a construction in Corollary \ref{cor:0dimimage} that shows that, while Lipschitz light mappings cannot decrease Lipschitz dimension, they can increase it arbitrarily, and in fact that every compact doubling metric space is the image under a Lipschitz light map of a space with Lipschitz dimension $0$.

In our study of David--Semmes regular mappings in subsection \ref{subsec:DSregular}, we also obtain in Corollary \ref{cor:DS} a necessary condition for the existence of non-degenerate Lipschitz maps between certain spaces.

Throughout the paper, we also include a number of questions that we consider worth studying.

\subsection{Remarks on the definition of Lipschitz light}\label{subsec:discrepancy}

Before proceeding further, we remark briefly on a discrepancy between our definition of Lipschitz light in Definition \ref{def:LL2} and \cite[Definition 1.14]{CK13_inverse}.

In \cite[Definition 1.14]{CK13_inverse}, a Lipschitz map $f\colon X\rightarrow Y$ between metric spaces is called Lipschitz light if there is $C>0$ such that, for every bounded subset $W\subset Y$, the $\diam(W)$-components of $f^{-1}(W)$ have diameter at most $C\diam(W)$.

Our Definition \ref{def:LL2} and \cite[Definition 1.14]{CK13_inverse} are equivalent if $Y=\RR^n$ ($n\geq 1$), but are not equivalent in general, as Remarks \ref{rmk:LL1} and \ref{rmk:LL2} now show. 

\begin{remark}\label{rmk:LL1}
It is clear that if a mapping satisfies Definition \ref{def:LL2}, then it satisfies \cite[Definition 1.14]{CK13_inverse}. If $n\geq 1$ and $Y=\RR^n$, it is not hard to show that the converse holds as well. Indeed, if $f\colon X\rightarrow \RR^n$ satisfies \cite[Definition 1.14]{CK13_inverse} and $W\subseteq \RR^n$ has $\diam(W)\leq r$, then one may find a point $x\in\RR^n$ such that $W'=W\cup\{x\}$ has $\diam(W')=r$. Any $r$-component of $f^{-1}(W)$ lies in an $r$-component, (i.e., a $\diam(W')$-component) of $f^{-1}(W')$, and hence has diameter at most $Cr$. 
\end{remark}

\begin{remark}\label{rmk:LL2}
In general, a mapping may satisfy \cite[Definition 1.14]{CK13_inverse} and not Definition \ref{def:LL2}, as the following example shows. Let $X=[0,1]\times(2\mathbb{Z})$, $Y=[0,1]$, and $f$ simply be the projection to the first factor, mapping from $X$ to $Y$. Then $f$ satisfies \cite[Definition 1.14]{CK13_inverse}: Any $W\subseteq Y$ has $\diam(W)\leq 1$, so any $\diam(W)$-component of $f^{-1}(W)$ is simply an isometric copy of $W$ contained in some $[0,1]\times \{2n\}$.

However, this mapping fails Definition \ref{def:LL2} in the case $W=Y$ and $r=2$, since $f^{-1}(W)$ has $2$-paths of arbitrarily large diameter.
\end{remark}

For the remainder of this paper, we use Definition \ref{def:LL2} above as our definition of Lipschitz light, as it is better adapted to general metric space targets. We point out that, for the purposes of computing Lipschitz dimension on spaces with positive Lipschitz dimension, it does not matter which definition one takes (by Remark \ref{rmk:LL1}), and that Definition \ref{def:LL2} is in any case the one used in Section 11 of \cite{CK13_inverse}.

\subsubsection*{Acknowledgments}
The author would like to thank Bruce Kleiner for helpful discussions concerning Lipschitz dimension a number of years ago, especially his pointer that Theorem \ref{thm:Carnot} should hold. This work was partially supported by the National Science Foundation under Grant No. DMS-1758709.

\section{Notation and Definitions} \label{sec:notation}

\subsection{Basic metric space notions}
We write $(X,d)$ for a metric space, which in this paper will generally be a complete metric space. Often, if the metric $d$ is understood from context, we denote it simply by $X$, and we often use the same symbol $d$ to denote the metric on different spaces. A pointed metric space is simply a pair $(X,x)$ consisting of a metric space $X$ and a point $x\in X$.

We denote open and closed balls in a metric space $X$ by
$$ B(x,r) = \{y\in X: d(y,x)<r\} \text{ and } \overline{B}(x,r) = \{y\in X: d(y,x) \leq r\}.$$
If we wish to emphasize the ambient space $X$ in which the ball is taken, we may write $B_X(x,r)$. If $\lambda>0$ and $B=B(x,r)$, it is convenient to write $\lambda B$ for $B(x,\lambda r)$.

The diameter of a set $E$ in a metric space $X$ is
$$ \diam(E) = \sup\{d(x,y): x,y\in E\}.$$

The distance between two sets $E$, $F$ in a metric space $X$ is
$$ \dist(E,F) = \inf\{d(x,y): x\in E, y\in F\}.$$
If one of these sets happens to be a single point, say $E=\{p\}$, then we write $\dist(p,F)$ rather than $\dist(\{p\},F)$.

If $E$ is a subset of a metric space $X$ and $r>0$, then the open and closed $r$-neighborhoods of $E$ in $X$ are
$$ N_r(E) = \{y\in X: \dist(y,E) < r\} \text{ and } \overline{N}_r(E)= \{y\in X: \dist(y,E) \leq r\}.$$

For $\epsilon>0$, an $\epsilon$-separated set in $X$ is a subset in which all mutual distances are at least $\epsilon$. An $\epsilon$-net $S$ in $X$ is a maximal $\epsilon$-separated set (which always exists by Zorn's lemma); in that case we have $X = N_{\epsilon}(S)$.

A metric space is \textit{proper} if all closed balls in the space are compact. A metric space is \textit{doubling} if there is a constant $N$ such that every ball in $X$ can be covered by $N$ balls of half the radius. This is a finite dimensionality condition; in fact, it is equivalent to the finiteness of the Assouad dimension defined below. Every complete, doubling metric space is automatically proper.

It is often useful to study the Cartesian product $X\times Y$ of two metric spaces $(X,d_X)$ and $(Y,d_Y)$. To fix a convention, unless otherwise noted, we take the metric on $X\times Y$ to be the $\ell_\infty$ combination of the metrics on the factors:
$$ d((x,y),(x',y')) = \max\{ d_X(x,x'), d_Y(y,y') \}.$$
Of course, this choice of product metric $d$ is bounded above and below by constant multiples of any of the other natural $\ell_p$ combinations of the two metrics.

In Section \ref{sec:Cheeger}, we will need the notion of a \textit{metric measure space}, which for us is a complete metric space $X$ equipped with a finite Radon measure $\mu$. A metric measure space is \textit{doubling} if the measure $\mu$ is doubling, meaning that there is a constant $C\geq 0$ such that
$$ \mu(2B) \leq C \mu(B)$$
for all balls $B$ in $X$. In particular, this implies that $X$ is a doubling metric space in the sense defined above \cite[Section 10.13]{He}.

\subsection{Mappings}
A function $f:X\rightarrow Y$ between two metric spaces is called \textit{Lipschitz} (or $L$-Lipschitz) if there is $L\geq 0$ such that 
$$ d(f(x),f(x')) \leq Ld(x,x') \text{ for all } x,x'\in X.$$
It is called \textit{bi-Lipschitz} (or $L$-bi-Lipschitz) if
$$ L^{-1}d(x,x') \leq d(f(x),f(x')) \leq Ld(x,x') \text{ for all } x,x'\in X.$$
A $1$-bi-Lipschitz map is called an \textit{isometric embedding}. 

A more general class than the bi-Lipschitz mappings is the class of \textit{quasisymmetric} mappings. An embedding $f\colon X \rightarrow Y$ is called quasisymmetric if there is a homeomorphism $\eta\colon [0,\infty)\rightarrow [0,\infty)$ such that
$$ d(x,a) \leq td(x,b) \text{ implies } d(f(x),f(a))\leq \eta(t) d(f(x),f(b)),$$
for all triples $a,b,x$ of points in $X$ and all $t\geq 0$. Quasisymmetric maps may wildly distort distances (in particular, they may not be Lipschitz), but in some sense they preserve ``shape''. See \cite{He} for an introduction to quasisymmetric mappings.

Other than the bi-Lipschitz mappings, another interesting sub-class of quasisymmetric mappings are the snowflake mappings. A mapping $f\colon X \rightarrow Y$ is called a \textit{snowflake} mapping (or an $\alpha$-snowflake mapping) if there are constants $\alpha \in (0,1]$ and $C>0$ such that
$$ C^{-1} d(x,y)^\alpha \leq d(f(x), f(y)) \leq Cd(x,y)^\alpha \text{ for all } x,y\in X. $$
A metric space $Z$ is called an $\alpha$-snowflake if it is the image of another metric space $X$ under an $\alpha$-snowflake mapping. Of course, this is equivalent to saying that $Z$ is bi-Lipschitz equivalent to the metric space $(X,d^\alpha)$.

The terminology ``snowflake'' arises from the fact that the standard von Koch snowflake curve in $\RR^2$, with the induced Euclidean metric, can be viewed as an $\alpha$-snowflake of $[0,1]$, where $\alpha^{-1}$ is the Hausdorff dimension of the snowflake.

A few other classes of mappings will be introduced in the paper as needed.

\subsection{Other notions of dimension}\label{subsec:dimensions}
We take the opportunity to recall the definitions of three other notions of dimension: the Hausdorff, Assouad, and Nagata dimensions. For more information about the Hausdorff and Assouad dimensions, we refer the reader to \cite[Sections 8.3 and 10.13]{He}, and for the Nagata dimension we refer the reader to \cite{LS}.

The \textit{$n$-dimensional Hausdorff measure} of a set $E$ in a metric space $X$ is
$$ \HH^n(E) = \lim_{\delta\rightarrow 0} \inf_{\{B_i\}}\sum_i \diam(B_i)^n,$$
where the infimum is over covers of $E$ by closed balls $B_i$ of diameter at most $\delta$. 
\begin{definition}
The \textit{Hausdorff dimension} of $X$ is 
$$ \dim_H(X) = \inf\{\alpha>0 : \HH^\alpha(E) = 0\} \in [0,\infty]$$
\end{definition}

\begin{definition}\label{def:assouad}
The \textit{Assouad dimension} $\dim_A(X)$ of a metric space $X$ is the infimum of all $\beta>0$ such that there is a constant $C$ for which every set of diameter $d$ can be covered by at most $C\epsilon^{-\beta}$ sets of diameter at most $\epsilon d$. 
\end{definition}
Equivalently, $\dim_A(X)$ can be defined as the infimum over all $\gamma>0$ such that there is a constant $C$ for which every ball of radius $r$ contains at most $C\epsilon^{-\gamma}$ $\epsilon r$-separated points.

Lastly, we define the Nagata dimension $\dim_N(X)$ of a metric space $X$. Call a family of subsets $\{B_i\}$ of $X$ \textit{$D$-bounded} if each $B_i$ has diameter $\leq D$. For $s>0$, the \textit{$s$-multiplicity} $\leq n$ of the family $\{B_i\}$ is the minimal $n$ such that every subset of $X$ with diameter $\leq s$ meets at most $n$ members of the family.

\begin{definition}\label{def:nagata}
The \textit{Nagata dimension} of $X$, which we denote $\dim_N X$, is the minimal integer $n$ with the following property: there exists $c>0$ such that, for all $s>0$, $X$ has a $cs$-bounded covering with $s$-multiplicity at most $n+1$.
\end{definition}
The Nagata dimension is clearly a quantitative analog of the Lebesgue covering definition of topological dimension, introduced at the start of subsection \ref{subsec:intro1}.

Each of the five notions of dimension defined above (topological, Lipschitz, Hausdorff, Assouad, and Nagata) is easily seen to be invariant under bi-Lipschitz homeomorphisms.

We have the following relationship between the above dimensions, for all separable metric spaces $X$:
\begin{equation}\label{eq:dimensionrelations1}
\dim_T(X) \leq \dim_N(X) \leq \dim_A(X)
\end{equation}
(see \cite[Theorem 2.2]{LS} and \cite[Theorem 1.1]{LDR}), and
\begin{equation}\label{eq:dimensionrelations2}
\dim_T(X) \leq \dim_H(X) \leq \dim_A(X)
\end{equation}
(see \cite[Theorem 8.13 and Exercise 10.6]{He}). Each of the above inequalities may be strict; we refer the reader to the references above for examples.

In Section \ref{sec:relation}, we explain where $\dim_L(X)$ does (and does not) fit into the above lists.

\section{Products and unions}\label{sec:productsunions}
We begin by studying the behavior of Lipschitz dimension under products and unions.

\begin{proposition}\label{lem:products}
Let $X$ and $Y$ be metric spaces with 
$$\dim_L (X) \leq m \text{ and } \dim_L (Y) \leq n.$$
Then 
$$ \dim_L (X\times Y) \leq m+n.$$
\end{proposition}
\begin{proof}
Let $Z=X\times Y$ and write $\pi_X$ and $\pi_Y$ for the projections to the two factors. 

Let $f:X\rightarrow \RR^m$ and $g:Y\rightarrow\RR^n$ be Lipschitz light mappings. Let
$$ F=(f,g):Z \rightarrow \RR^{m+n}.$$

Fix $W\subseteq \RR^{m+n}$ with $\diam(W)\leq r$.  We write $\pi_{\RR^m}\colon \RR^{m+n}\rightarrow \RR^m$ for the projection to the first $m$ coordinates and $\pi_{\RR^n}\colon \RR^{m+n}\rightarrow \RR^n$ for the projection to the last $n$ coordinates.

Let $A$ be an $r$-component of $F^{-1}(W)\subseteq Z$. Note that
$$ f(\pi_X(A)) = \pi_{\RR^m}(F(A)) \subseteq \pi_{\RR^m}(W)$$
and
$$ g(\pi_Y(A)) = \pi_{\RR^n}(F(A)) \subseteq \pi_{\RR^n}(W).$$

If $P$ is any $r$-path in $A$, then $\pi_X(P)$ and $\pi_Y(P)$ are $r$-paths in $f^{-1}(\pi_{\RR^m}(W))$ and $g^{-1}(\pi_{\RR^n}(W))$, respectively.

It follows that $\pi_X(P)$ and $\pi_Y(P)$ have diameters controlled by $Cr$, where $C$ is the maximum of the Lipschitz light constants of $f$ and $g$. Thus, $P$ has diameter controlled by $Cr$. Since $P$ was an arbitrary $r$-path in $A$, $\diam(A)\leq Cr$. This proves that $F$ is Lipschitz light, and hence that $\dim_L Z \leq m+n$.
\end{proof}
The inequality in Lemma \ref{lem:products} is of course sometimes attained (e.g., by $\RR\times\RR$) but it may be strict in some cases. The same example following \cite[Theorem 2.6]{LS} shows this: Let $X=\mathbb{Z}$ and $Y=[0,1]$. Then it is easy to see that $\dim_L X = \dim_L Y = 1$. 

However, $\dim_L (X\times Y) = 1$ as well. To show this, we show that the map $f\colon X\times Y\rightarrow \RR$ defined by $f(n,t)=2n+t$ is a bi-Lipschitz embedding: It is clearly Lipschitz. If $n=m$, then $|f(n,t)-f(m,s)|=|t-s|=d((n,t),(m,s))$. Otherwise, 
$$|f(n,t)-f(m,s)|\geq 2|n-m|-|t-s|\geq |n-m| = d((n,t),(m,s)).$$

Next we study unions. While we are able to show that a finite union of spaces with finite Lipschitz dimension has finite Lipschitz dimension, we do not appear to obtain the sharp bound.

\begin{proposition}\label{lem:union}
Let $Z$ be a metric space that can be written as a union $Z=X\cup Y$. Then 
$$ \dim_L (Z) \leq \dim_L(X)+\dim_L(Y).$$
\end{proposition}
\begin{proof}
Write $m=\dim_L(X)$ and $n=\dim_L(Y)$. Of course, if either is infinite, then there is nothing to prove.

Let $f:X\rightarrow \RR^m$ and $g:Y\rightarrow \RR^n$ be Lipschitz light. We may assume that both are Lipschitz light with constant $C\geq 1$. By McShane's extension theorem (\cite[Theorem 6.2]{He}), we may extend both mappings to Lipschitz mappings defined on all of $Z$, though of course they will not necessarily be Lipschitz light on the entire domain $Z$.

Let $F:Z\rightarrow \RR^{n+m}$ be defined by $F(z) = (f(z), g(z))$. We claim that $F$ is Lipschitz light.

Let $W\subseteq\RR^{n+m}$ be a set of diameter at most $r>0$, and let 
$$P=(x_1, x_2, \dots, x_k)$$
be an $r$-path in $F^{-1}(W)\subseteq Z$. Without loss of generality, assume that $x_1\in X$.

We make the following immediate observation: 
\begin{equation}\label{eq:unionobs}
\text{If $A\geq 1$, then any $Ar$-path $Q$ contained in $P\cap X$ or $P\cap Y$ has diameter at most $CAr$}
\end{equation}
Indeed, for such a path $Q$, either $f$ or $g$ maps it into a set of diameter at most $r\leq Ar$, and both these maps are Lipschitz light on their respective original domains $X$ and $Y$.

We now define a sub-path $P' \subseteq P$ as follows. 

Let $i_1=1$. For each $j>1$, inductively set $i_j$ to be the smallest index greater than $i_{j-1}$ such that $x_{i_j}\in X$. Continue this until there is no such index $i_j$. We obtain a sub-path
$$ P' = (x_{i_1}, x_{i_2}, \dots, x_{i_\ell})\subseteq P.$$
Observe that if $i_j> i_{j-1}+1$, then the entire sub-path of $P$ between from index $i_{j-1}+1$ to $i_j - 1$ is contained in $Y$, since it is disjoint from $X$. Thus, by \eqref{eq:unionobs}, the diameter of this sub-path is at most $Cr$. The same holds for the sub-path between index $i_\ell$ and the last index $k$, if it so happens that $i_\ell < k$.

Thus,
$$ d(x_{i_{j-1}}, x_{i_j}) \leq (C+2)r \text{ for each } j,$$
i.e., $P'$ is a $(C+2)r$-path, and moreover
$$ P \subseteq \overline{N}_{Cr}(P').$$

Since $P'$ is a $(C+2)r$-path that is entirely contained in $X$, it follows again from \eqref{eq:unionobs} that
$$\diam(P')\leq C(C+2)r.$$
Hence,
$$\diam(P) \leq \diam(P')+2(C+2)r \leq (C+2)^2r.$$
Thus, $F$ is Lipschitz light and so $\dim_L(Z) \leq n+m$.
\end{proof}
Of course, Lemma \ref{lem:union} implies that any finite union of spaces with finite Lipschitz dimension has finite Lipschitz dimension. 

If true, the natural bound in Proposition \ref{lem:union} would be to replace the sum by the maximum:
\begin{question}
If $Z=X\cup Y$, is $\dim_L(Z) \leq \max\left(\dim_L(X), \dim_L(Y)\right)$?
\end{question}

\section{Gromov-Hausdorff limits and weak tangents}\label{sec:GH}

\subsection{Convergence of metric spaces}
We will use the notion of convergence of ``mapping packages'', a version of Gromov-Hausdorff convergence, that is described in Chapter 8 of \cite{DS97}. Expositions of this material are also given in \cite{Keith} and \cite[Section 2.1]{GCD16}.

\begin{definition}\label{sets}
We say that a sequence $\{F_j\}$ of non-empty, closed subsets of some Euclidean space $\mathbb{R}^N$ converges to a non-empty closed set $F\subseteq\mathbb{R}^N$ if
$$ \lim_{j\rightarrow\infty} \sup_{x\in F_j\cap B(0,R)} \dist(x,F) = 0 $$
and
$$  \lim_{j\rightarrow\infty} \sup_{y\in F\cap B(0,R)} \dist(y,F_j) = 0 $$
for all $R>0$. 
\end{definition}

We now move on to defining convergence of mappings.
\begin{definition}\label{maps}
Suppose $\{F_j\}$ is a sequence of closed sets converging to a closed set $F$ in $\mathbb{R}^N$ as in the previous definition. Let $Y$ be a metric space and $f_j:F_j\rightarrow Y$, $f:F\rightarrow Y$ be mappings. We say that $\{f_j\}$ converges to $f$ if for each sequence $\{x_j\}$ in $\mathbb{R}^N$ such that $x_j\in F_j$ for all $j$ and $x_j\rightarrow x\in F$, we have that
$$ \lim_{j\rightarrow\infty} f_j(x_j) = f(x). $$
\end{definition}

We have the following compactness statements for these notions of convergence:
\begin{lemma}[\cite{DS97}, Lemmas 8.2 and 8.6]\label{lem:Rncompactness}
Let $\{F_j\}$ be a sequence of non-empty, closed subsets of $\RR^n$ that all intersect a fixed ball $B(0,r)$. Let $f_j\colon F_j \rightarrow \RR^m$ be $L$-Lipschitz mappings. 

Then there is a subsequence along which $\{F_j\}$ converges to a non-empty, closed subset $F$ of $\RR^n$ (in the sense of Definition \ref{sets}) and $\{f_j\}$ converges to a $L$-Lipschitz mapping $f\colon F\rightarrow \RR^m$ (in the sense of Definition \ref{maps}).
\end{lemma}

Now we begin to define convergence for general metric spaces and mappings.

\begin{definition}\label{spaceconvergence}
A sequence of pointed metric spaces $\{(X_j, d_j, p_j)\}$ converges to a pointed metric space $(X,d,p)$ if the following conditions hold. There exists $\alpha\in(0,1]$, $N\in\mathbb{N}$, and $L$-bi-Lipschitz embeddings $e_j:(X_j, d_j^\alpha)\rightarrow\mathbb{R}^N$, $e:(X,d^\alpha)\rightarrow\mathbb{R}^N$ with $e_j(p_j) = e(p) = 0$ for all $j$. Furthermore, we require that $e_j(X_j)$ converge to $e(X)$ in the sense of Definition \ref{sets}, and that the real-valued functions $d_j(e_j^{-1}(x), e_j^{-1}(y))$ defined on $e_j(X_j) \times e_j(X_j)$ converge to $d(e^{-1}(x), e^{-1}(y))$ on $e(X) \times e(X)$ in the sense of Definition \ref{maps}.
\end{definition}

We will only use Definition \ref{spaceconvergence} when the metric spaces $\{(X_j, d_j)\}$ and $(X,d)$ are uniformly doubling. In that case, the embeddings $e_j$ and $e$ can always be found, by Assouad's embedding theorem (see \cite{He}, Theorem 12.2).

\begin{definition}\label{mappingpackage}
A \textit{mapping package} consists of a pair of pointed metric spaces $(M, d_M, p)$ and $(N, d_N, q)$ as well as a mapping $g:M\rightarrow N$ such that $g(p)=q$. It is written $\left( (M, d_M, p), (N,d_N,q), g\right)$.
\end{definition}
We slightly abuse notation and call a mapping package ``doubling'' if the underlying spaces are both doubling, and ``uniformly doubling'' if all underlying spaces are doubling with the same doubling constant.

\begin{definition}\label{mappingpackageconvergence}
A sequence of mapping packages $\{((X_j, d_j, p_j), (Y_j, \rho_j, q_j), h_j)\}$ is said to converge to another mapping package $((X, d, p), (Y, \rho, q), h)$ if the following conditions hold. The sequences $\{(X_j, d_j, p_j)\}$ and $\{(Y_j, \rho_j, q_j\}$ converge to $(X,d,p)$ and $(Y,\rho, q)$, respectively, in the sense of Definition \ref{spaceconvergence}. Furthermore, the maps $g_j \circ h_j \circ f_j^{-1}$ converge to $g \circ h \circ f^{-1}$ in the sense of Definition \ref{maps}, where $f_j, g_j, f, g$ are the embeddings of Definition \ref{spaceconvergence}.
\end{definition}

We take this opportunity to remark that the limit of a sequence of mapping packages is unique up to isometry; see \cite[Lemma 8.20]{DS97}. That is, two limits of the same sequence of mappings packages are isometrically equivalent by an isometry that preserves base points and intertwines the mappings.

We often use $\rightarrow$ notation to indicate convergence of a sequence of pointed metric spaces or mapping packages, e.g., $(X_j, d_j, p_j)\rightarrow (X,d,p)$.

The following proposition is a special case of Lemma 8.22 of \cite{DS97}.

\begin{proposition}\label{sublimit}
Let $\{((X_j, d_j, p_j), (Y_j, \rho_j, q_j), h_j)\}$ be a sequence of mapping packages, in which all the metric spaces are complete and uniformly doubling, and in which the maps $h_j$ are equicontinuous and uniformly bounded on bounded sets and satisfy $h_j(p_j)=q_j$. Then there exists a mapping package $((X, d, p), (Y, \rho, q), h)$ that is the limit of a subsequence of $\{((X_j, d_j, p_j), (Y_j, \rho_j, q_j), h_j)\}$.
\end{proposition}
Here, the assumption that the $\{h_j\}$ are \textit{equicontinuous and uniformly bounded on bounded sets} means that for each $R>0$ and $\epsilon>0$, there is $\delta>0$ such that 
$$\rho_j(h_j(x), h_j(y)) < \epsilon \text{ for all } x,y \in B_{X_j}(p_j, R) \text{ with } d_j(x,y)<\delta$$
and
$$ \sup_j \sup_{x\in B_{X_j}(p_j,R)} \rho_j(h_j(x),q_j) < \infty.$$
In particular, this assumption is satisfied when the $h_j$ are Lipschitz or snowflake maps with constants controlled independent of $j$, which is how we will always use this result.

We will now describe some consequences of the convergence of a sequence of mapping packages, which are Lemmas 8.11 and 8.19 of \cite{DS97}.

\begin{proposition}\label{almostisos1}
Suppose a sequence of pointed metric spaces $\{(X_k, d_k, p_k)\}$ converges to the pointed metric space $(X,d,p)$, in the sense of Definition \ref{spaceconvergence}.

Then there exist (not necessarily continuous) mappings $\phi_k:X\rightarrow X_k$ and $\psi_k: X_k \rightarrow X$ such that:
\begin{enumerate}[(i)]
\item\label{almostisos1i} For all $k$, $\phi_k(p) = p_k$ and $\psi_k(p_k) = p$.

\item\label{almostisos1ii} For all $R>0$,
$$ \lim_{k\rightarrow\infty} \sup\{d_X(\psi_k(\phi_k(x), x) : x\in B_X(p, R) \} = 0$$
and
$$ \lim_{k\rightarrow\infty} \sup\{d_{X_k}(\phi_k(\psi_k(x), x) : x\in B_{X_k}(p_k, R) \} = 0.$$

\item For all $R>0$,
$$ \lim_{k\rightarrow\infty} \sup\{|d_{X_k}(\phi_k(x) , \phi_k(y)) - d_X(x,y)| : x,y\in B_X(p, R) \} = 0$$
and
$$ \lim_{k\rightarrow\infty} \sup\{|d_{X}(\psi_k(x) , \psi_k(y)) - d_{X_k}(x,y)| : x,y\in B_{X_k}(p_k, R) \} = 0.$$
\end{enumerate}
\end{proposition}

\begin{proposition}\label{almostisos}
Suppose a sequence of mapping packages 
$$\{((X_k, d_k, p_k), (Y_k, \rho_k, q_k), h_k)\}$$
 converges to a mapping package 
$$((X, d, p), (Y,\rho,q),h),$$
where the mappings $h_k$ are uniformly Lipschitz and satisfy $h_k(p_k)=q_k$. Then there exist (not necessarily continuous) mappings $\phi_k:X\rightarrow X_k$ and $\psi_k: X_k \rightarrow X$ satisfying exactly the conditions of Proposition \ref{almostisos1}, and mappings $\sigma_k:Y\rightarrow Y_k$ and $\tau_k: Y_k \rightarrow Y$ satisfying the analogous properties of Proposition \ref{almostisos1}, such that in addition we have the following:

For all $x\in X$, 
\begin{equation}\label{eq:almostisoseq}
\lim_{k\rightarrow\infty} \tau_k(h_k(\phi_k(x))) = h(x)
\end{equation}
and this convergence is uniform on bounded subsets of $X$.
\end{proposition}

The following lemma is needed to ensure that subsets of our spaces converge. It is a simple consequence of \cite[Lemma 8.31]{DS97}.

\begin{lemma}\label{subsetconvergence}
Suppose that $\{(X_k, d_k, p_k)\}$ is a sequence of pointed metric spaces that converges to the pointed metric space $\{(X, d, p)\}$ in the sense of Definition \ref{spaceconvergence}. Let $\{F_k\}$ be a sequence of nonempty closed sets with 
$$p_k \in F_k\subset X_k \text{ for each } k.$$ 
Then we can pass to a subsequence to get convergence of $(F_k, d_k, p_k)$ to $(F,d,p)$, where $F$ is a nonempty closed subset of $X$.
\end{lemma}

Lastly, we record the following basic facts about the preservation of mapping properties under limits. 

\begin{lemma}\label{lem:GHmaps}
Let $\{\left((X_k, d_k, p_k), (Y_k, \rho_k,q_k), f_k\right)\}$ be a sequence of mapping packages converging to\\ $\left((X, d, p), (Y, \rho,q), f\right).$
\begin{enumerate}[(i)]
\item If all the $f_k$ are $L$-Lipschitz, then so is $f$.
\item If all the $f_k$ are $L$-bi-Lipschitz, then so is $f$.
\item If all the $f_k$ are $\alpha$-snowflake maps with constant $C$, then so is $f$.
\item If all the $f_k$ are surjective $\alpha$-snowflake maps with constant $C$, then $f$ is a surjective $\alpha$-snowflake map.
\end{enumerate}
\end{lemma}
\begin{proof}
The first two statements in the lemma are given in \cite[Lemma 8.20]{DS97}, and the third is easy to verify by the same means. The fourth follows by passing to a subsequence along which the packages
$$ \{\left((Y_k, \rho_k,q_k), (X_k, d_k^\alpha, p_k), (f_k)^{-1}\right)\}$$
converge as well, which we can do by (iii).
\end{proof}

\subsection{Tangents and weak tangents}
We can now define the notion of tangent and weak tangent to a space or mapping package.

\begin{definition}
If $(X,d)$ is a metric space, a \textit{weak tangent} of $X$ is any limit of pointed metric spaces of the form $(X,\lambda_k d, x_k)$, where $\lambda_k>0$ and $x_k\in X$.

If $f:(X,d) \rightarrow (Y,\rho)$ is a mapping, a \textit{weak tangent} of $f$ is any limit of mapping packages of the form
$$ \left( (X,\lambda_k d, x_k), (Y, \lambda \rho_k, f(x_k)),f\right),$$
where $\lambda_k>0$ and $x_k\in X$ are arbitrary.

We denote the collection of weak tangents of $X$ or $f$ by $\WTan(X)$ or $\WTan(f)$, respectively.
\end{definition}

As a special case of the notion of weak tangent, one may force the base points $x_k$ to be fixed and the sequence of scales to tend to infinity, corresponding to the notion of ``blowing up'' the space at a given point. This is the notion of a tangent.

\begin{definition}
If $(X,d)$ is a metric space and $x\in X$, a \textit{ tangent} of $X$ at $x$ is any limit of pointed metric spaces of the form $(X,\lambda_k d, x)$, where $\lambda_k\rightarrow \infty$.

If $f:(X,d) \rightarrow (Y,\rho)$ is a mapping and $x\in X$, a \textit{tangent} of $f$ is any limit of mapping packages of the form
$$ \left( (X,\lambda_k d, x), (Y, \lambda_k \rho, f(x)),f\right),$$
where $\lambda_k\rightarrow \infty$.

We denote the collection of tangents of $X$ at $x$ by $\Tan(X,x)$, and the collections of tangents of $f$ at $x$ by $\Tan(f,x)$.
\end{definition}

Of course, tangents are always weak tangents. Proposition \ref{sublimit} guarantees that a doubling metric space has at least one tangent at each of its points, and that a Lipschitz mapping has a tangent at each point of its domain.

If $F$ is a non-empty, closed subset of $\RR^n$, one may wish to pass to a tangent or weak tangent of $F$ inside $\RR^n$, rather than viewing it simply as an abstract metric space. We have the following simple consequence of the lemmas above.

\begin{lemma}\label{lem:canonicaltangent}
Let $F\subseteq \RR^n$ be a non-empty, closed set, and let $f\colon F\rightarrow \RR^m$ be a Lipschitz mapping. Let $\{x_j\}$ be a sequence of points in $F$ and $\{\lambda_j\}$ a sequence of positive numbers. Then the following statements hold:

\begin{enumerate}[(i)]
\item We may pass to a subsequence along which the sets 
\begin{equation}\label{eq:ct1}
\lambda_j(F - x_j)
\end{equation}
and the mappings
\begin{equation}\label{eq:ct2}
 z\mapsto \lambda_j(f\left( \lambda_j^{-1}z + x_j\right)-f(x_j)).
\end{equation}
converge to a set $\hat{F}\subseteq \RR^n$ and a mapping $\hat{f}\colon \hat{F}\rightarrow \RR^m$, in the sense of Definitions \ref{sets} and \ref{maps}.

Furthermore, $(\hat{F},0)$ is in $\WTan(F)$ and the mapping package 
$$((\hat{F},0),(\RR^m,0),\hat{f})$$
is in $\WTan(f)$.

\item Conversely, if $((Z,p), (\RR^m,0),h)\in \WTan(f)$ arises from the choice of points $\{x_j\}$ and scales $\{\lambda_j\}$, then, after passing to a subsequence, the sets in \eqref{eq:ct1} converge to a set $\hat{F}$ isometric to $(Z,p)$ and the mappings in \eqref{eq:ct2} converge to a mapping on $\hat{F}$ that agrees with $h$, up to composition with an isometry.
\end{enumerate}
\end{lemma}
\begin{proof}
This is an immediate consequence of Lemma \ref{lem:Rncompactness} and Lemma \ref{sublimit}, and the fact that limits of isometric spaces and mappings are themselves isometric. (See \cite[Lemma 8.12]{DS97}.)
\end{proof}

\subsection{Gromov-Hausdorff limits of Lipschitz light mappings}
In this subsection, we study the Gromov-Hausdorff convergence properties of Lipschitz light mappings, culminating in a characterization result, Theorem \ref{prop:LLGH}, and corresponding consequences for Lipschitz dimension. 

We begin by establishing the persistence of the Lipschitz light property during Gromov-Hausdorff convergence.
\begin{proposition}\label{prop:LLlimit}
Let $\{\left( (X_k,d_k,p_k), (Y_k,\rho_k, q_k), f_k\right)\}$ be a sequence of complete, uniformly doubling mapping packages converging to  $\{((X,d,p), (Y,\rho,q), f)\}$. Assume that each $f_n$ is Lipschitz light with constant $C$, independent of $k$. Then $f$ is Lipschitz light with constant $C$.
\end{proposition}
\begin{proof}
To begin, the map $f$ is $C$-Lipschitz by Lemma \ref{lem:GHmaps}.

Consider a sequence of mapping packages as in the Proposition. We have ``almost-isometries''  $\phi_k:X\rightarrow X_k$, $\psi_k: X_k \rightarrow X$, $\sigma_k:Y\rightarrow Y_k$, and $\tau_k: Y_k \rightarrow Y$ as in Proposition \ref{almostisos}.

Fix $r>0$ and $W\subseteq Y$ with $\diam(W)\leq r$. Let $U$ be an $r$-component of $f^{-1}(W)\subseteq X$, and $P$ an arbitrary $r$-path in $U$.

We can choose $R>0$ large enough so that $P\subseteq B(p,R/2)$ and $P_k:=\phi_k(P)\subseteq B(p_k,R/2)$ for each $k\in\mathbb{N}$. Let $\epsilon\in (0,r)$ be arbitrary. We may then choose $k\in\mathbb{N}$ sufficiently large so that all distortions of $\phi_k,\psi_k,\sigma_k,\tau_k$ are less than $\epsilon$. In other words, 
$$\sup\{d(\psi_k(\phi_k(x), x) : x\in B_X(p, R) \} <\epsilon,$$
$$ \sup\{d_k(\phi_k(\psi_k(x), x) : x\in B(p_k, R) \} <\epsilon,$$
$$ \sup\{|d_{k}(\phi_k(x) , \phi_k(y)) - d(x,y)| : x,y\in B(p, R) \} <\epsilon,$$
and
$$ \sup\{|d(\psi_k(x) , \psi_k(y)) - d_{k}(x,y)| : x,y\in B(p_k, R) \} <\epsilon,$$
with analogous properties for $\tau_k$ and $\sigma_k$. We may furthermore ensure that
$$| \tau_k(f_k(\phi_k(x))) - f(x)| < \epsilon \text{ for all } x\in B(p,R).$$

Then $P_k=\phi_k(P)$ is an $(r+2\epsilon)$-path in $X_k$. Moreover,
$$\diam(f_k(P_k)) \leq \diam(\tau_k(f_k(P_k))) + 2\epsilon \leq \diam f(P) + 4\epsilon \leq r+4\epsilon.$$
Since $f_k$ is Lipschitz light with constant $C$ and $f_k(P_k)$ is a set of diameter at most $r+4\epsilon$, it follows that
$$ \diam(P_k) \leq C(r+4\epsilon).$$
Lastly, we have
$$ \diam(P) \leq \diam(P_k) + 2\epsilon \leq C(r+4\epsilon)+2\epsilon.$$
Letting $\epsilon$ tend to $0$, we get that 
$$ \diam(P) \leq Cr,$$
which proves that $f$ is Lipschitz light with constant $C$.
\end{proof}

\begin{corollary}\label{cor:LLtan}
If $X$ and $Y$ are complete and doubling, and $f:X\rightarrow Y$ is Lipschitz light with constant $C$, then so is each $\hat{f}\in \WTan(f)$.
\end{corollary}
\begin{proof}
We need only observe that the mapping
$$ f:(X,\lambda d)\rightarrow (Y,\lambda\rho)$$
is also Lipschitz light with constant $C$, for each $\lambda>0$, and then apply Proposition \ref{prop:LLlimit}.
\end{proof}

As an immediate consequence, we show that Lipschitz dimension cannot increase when passing to weak tangents.

\begin{corollary}\label{cor:Ldimtan}
If $X$ is doubling, $\dim_L (X) \leq n$ and $(Z,z)\in \WTan(X)$, then $\dim_L (Z) \leq n$.
\end{corollary}
\begin{proof}
If $f:X\rightarrow \RR^n$ is Lipschitz light, then by Proposition \ref{sublimit} there is a mapping $\hat{f}\in \WTan(f)$ mapping from $Z$ to $\RR^n$. The mapping $\hat{f}$ is Lipschitz light by Corollary \ref{cor:LLtan}.
\end{proof}

In fact, we can also characterize Lipschitz light mappings among all Lipschitz mappings by examining their weak tangents.

Before we do so, we need the following lemma.
\begin{lemma}\label{lem:pathconvergence}
Let 
$$ \{(X_n, d_n, p_n)\} \rightarrow (X,d,p)$$
be a converging sequence of complete, uniformly doubling pointed metric spaces. 

Suppose that for each $n$, there is a $\delta_n$-path $P_n \subseteq \overline{B}(p_n,1)\subseteq X_n$ containing $p_n$, and that $\delta_n\rightarrow 0$. 

Then, after passing to a subsequence,
$$ \{(P_n, d_n, p_n)\} \rightarrow (P,d,p),$$
where $P\subseteq X$ is compact and connected.
\end{lemma}
\begin{proof}
The existence of a subsequence under which $P_n$ converges to a compact set $P\subseteq X$ is assured by Lemma \ref{subsetconvergence}. Assume that $P$ is not connected. It follows that $P$ can be written as $A\cup B$, where $\epsilon:=\text{dist}(A,B)>0$.

Fix mappings $\phi_n:P\rightarrow P_n$ and $\psi_n: P_n \rightarrow P$ as in Proposition \ref{almostisos1}. By choosing $n$ large, we may ensure that 
$$\delta_n<\epsilon/10,$$
$$ |d_n(\phi_n(x), \phi_n(y)) - d(x,y)| < \epsilon/10,$$
and
$$ |d_n(\psi_n(\phi_n(x)),x)| < \epsilon/10$$
for all $x,y\in P$.

Fix $a\in A\subseteq P$ and $b\in B\subseteq P$. Then $\phi_n(a)$ and $\phi_n(b)$ are in $P_n$, so there is a $\delta_n$-path 
$$(\phi_n(a)=x_n^1, \dots, \phi_n(b)=x_n^m)\subseteq P_n$$ 
between them. The path
$$ (a, \psi_n(x_n^1), \psi_n(x_n^2), \dots, \psi_n(x_n^m), b) \subseteq P$$
is then a $\frac{3\epsilon}{10}$-path from $a$ to $b$ in $P$. But this is impossible, since $\text{dist}(A,B)=\epsilon$.
\end{proof}

\begin{theorem}\label{prop:LLGH}
Let $f:X\rightarrow Y$ be a Lipschitz mapping between complete, doubling metric spaces. Then the following are equivalent:
\begin{enumerate}[(i)]
\item\label{LLGHi} $f$ is Lipschitz light.
\item\label{LLGHii} Each weak tangent of $f$ is Lipschitz light.
\item\label{LLGHiii} Each weak tangent of $f$ is light.
\end{enumerate}
\end{theorem}
\begin{proof}
We have already shown that \eqref{LLGHi} implies \eqref{LLGHii} in Corollary \ref{cor:LLtan}. Since Lipschitz light maps are automatically light, \eqref{LLGHii} immediately implies \eqref{LLGHiii}. It remains to show that \eqref{LLGHiii} implies \eqref{LLGHi}.

Suppose that every weak tangent of $f$ is light, but that $f$ is not Lipschitz light. That means that that, for each $n\in\mathbb{N}$, we have
\begin{itemize}
\item a positive number $r_n$,
\item a subset $W_n \subseteq Y$ with $\diam(W_n) \leq r_n$, and
\item an $r_n$-path $P_n\subseteq f^{-1}(W_n)$ with $\diam(P_n)\geq nr_n$.
\end{itemize}
Let $x_n$ be the initial point of $P_n$. We consider the following sequence of mapping packages:
$$\left\{ \left( (X,\frac{1}{\diam P_n} d,x_n), (Y, \frac{1}{\diam P_n}\rho, f(x_n)), f \right) \right\}.$$
By Proposition \ref{sublimit}, a subsequence of this sequence converges to a mapping package
$$ \left\{ \left( (\hat{X}, \hat{d}, \hat{x}),(\hat{Y},\hat{\rho},\hat{f}(\hat{x})), f \right)\right\}$$
in $\WTan(f)$.

In the space $ (X,\frac{1}{\diam P_n} d,x_n)$, $P_n$ is a $\frac{1}{n}$-path of diameter exactly $1$. By passing to a further subsequence, we may assume that $P_n$ converges to a connected subset $\hat{P}\subset \hat{X}$, as in Lemma \ref{lem:pathconvergence}. Furthermore, $\diam(\hat{P})=1$, since $\diam(\hat{P_n})=1$ for each $n$.

On the other hand, $f(P_n)\subseteq W_n$ has diameter at most $\frac{1}{n}$ in $(Y, \frac{1}{\diam P_n}\rho)$, and therefore $\hat{f}(\hat{P})$ is a single point in $\hat{Y}$. 

Thus, $\hat{f}$ collapses the non-trivial connected set $\hat{P}$ to a point. It follows that $\hat{f}$ is not light, contradicting our assumption \eqref{LLGHiii}.
\end{proof}

\section{Lipschitz dimensions of various spaces}\label{sec:examples}

In this section, we use a variety of techniques to compute or bound the Lipschitz dimension of a number of spaces. The main concrete results of interest are Corollary \ref{cor:treebuilding} concerning trees and buildings, Corollary \ref{cor:Rnsnowflake} on snowflakes of Euclidean spaces, Theorem \ref{thm:Carnot} concerning Carnot groups, and Theorem \ref{theorem:ss} covering certain fractals in Euclidean space.

\subsection{Trees and Euclidean buildings}\label{subsec:treesbuildings}
In this section, we study two classes of non-positively curved spaces: metric trees and Euclidean buildings. A \textit{metric tree} is a geodesic metric space such that all geodesic triangles are degenerate. In other words, a metric tree is a space $T$ such that every two points $x,y\in T$ can be joined by a curve $\gamma_{xy}$ of length $d(x,y)$, and such that if $x,y,z\in T$ then $\gamma_{xz}\subseteq \gamma_{xy}\cup\gamma_{yz}.$ In particular, as in \cite{LS}, no compactness or local finiteness is assumed, so a metric tree may have arbitrarily large Hausdorff dimension, for example.

The definition of Euclidean building would take us rather far afield here, so we refer those who are interested to \cite[Section 6]{LPS} or \cite{KL} for details. The definition of Euclidean building will not directly enter our arguments in this section; rather we use only two results from \cite{LS} about trees and buildings.

In \cite[Lemma 3.1]{LS}, Lang and Schlichenmaier study mappings that satisfy certain technical conditions (those in Lemma \ref{lem:LSlemma} below), and show that such mappings cannot decrease Nagata dimension. They then construct such mappings from trees and buildings into $\RR^n$, in order to bound the Lipschitz dimensions of these spaces.

We show that the mappings studied by Lang and Schlichenmaier are in fact Lipschitz light.

\begin{lemma}\label{lem:LSlemma}
Suppose $f\colon X\rightarrow Y$ is $1$-Lipschitz and $h\colon X\times[0,\infty)\rightarrow X$ are mappings with the following three properties, for some $\lambda,\mu>0$:
\begin{enumerate}[(i)]
\item Whenever $C\subset Y$ is non-empty and bounded, there exists $y\in Y$ such that
$$ f^{-1}(C) \subset \overline{N}_{\lambda \diam (C)}(f^{-1}(y)). $$
\item For all $x\in X$ and $t\geq 0$, $d(h(x,t),x)\leq t$.
\item If $f(x)=f(x')$ and $t\geq \mu d(x,x')$, then $h(x,t) = h(x',t)$.
\end{enumerate}
Then $f$ is Lipschitz light.
\end{lemma}
\begin{proof}
Consider any $C \subset Y$ with $\diam C\leq r$. Consider any $r$-path $(x_1, \dots, x_n)$ in $f^{-1}(C)$. By (i), there is a corresponding $(1+2\lambda)r$-path $(z_1, \dots, z_n)\subset f^{-1}(y)$ for some $y\in Y$, with $d(x_i, z_i)\leq \lambda r$ for each $i$.

By (iii), we see that
$$ h(z_i, \mu(1+2\lambda)r) = h(z_{i+1}, \mu(1+2\lambda)r) $$
for each $i\in\{1,\dots, n-1\}$. So
$$ h(z_i, \mu(1+2\lambda)r) = h(z_j, \mu(1+2\lambda)r) $$
for each $i,j\in\{1,\dots, n\}$. Thus, there is a point $p\in X$ with 
$$ h(z_i, \mu(1+2\lambda)r) = p \text{ for each } i\in\{1,\dots,n\}.$$

It follows from (ii) that
$$ d(z_i, p) = d(z_i, h(z_i, \mu(1+2\lambda)r)) \leq \mu(1+2\lambda)r \text{ for each } i \in\{1,\dots,n\},$$
and so
$$ \diam(\{z_1, \dots, z_n\}) \leq 2\mu(1+2\lambda)r. $$
Therefore,
$$ \diam(\{x_1, \dots, x_n\}) \leq 2\mu(1+2\lambda)r + 2\lambda r = 2(\mu+2\mu\lambda+\lambda)r $$
and so $f$ is Lipschitz light.
\end{proof}

It is a trivial observation that if $f:X\rightarrow Y$ is Lipschitz light, then $\dim_L (X) \leq \dim_L (Y)$. (One simply observes that the composition of $f$ with any Lipschitz light mapping from $Y$ to some $\RR^n$ is also Lipschitz light.) As an immediate corollary, if $X$ and $Y$ are as in Lemma \ref{lem:LSlemma}, then $\dim_L(X) \leq \dim_L(Y)$. We now observe that Lang and Schlichenmaier in fact prove the following en route to Theorems 3.2 and 3.3 of \cite{LS}.

\begin{theorem}[Lang-Schlichenmaier \cite{LS}, Theorems 3.2 and 3.3]\label{thm:LSthm}
Let $T$ be a metric tree and let $X$ be a Euclidean building of rank $n$. Then
\begin{enumerate}[(i)]
\item There are maps $f_T:T\rightarrow \RR$ and $h_T:T\times \RR \rightarrow T$ satisfying the assumptions of Lemma \ref{lem:LSlemma}.
\item There are maps $f_X:X\rightarrow \RR^n$ and $h_X:X\times \RR \rightarrow X$ satisfying the assumptions of Lemma \ref{lem:LSlemma}.
\end{enumerate}
\end{theorem}

As a consequence, we have:

\begin{corollary}\label{cor:treebuilding}
Let $X$ be a product of $n$ (non-trivial) metric trees or a Euclidean building of rank $n$. Then the Lipschitz dimension of $X$ is $n$.
\end{corollary}
\begin{proof}
By Lemma \ref{lem:LSlemma} and Theorem \ref{thm:LSthm}, a Euclidean building of rank $n$ has Lipschitz dimension at most $n$, and a metric tree has Lipschitz dimension at most $1$.

Since a Euclidean building of rank $n$ contains an isometric copy of $\RR^n$, its dimension must be $n$.

By Lemma \ref{lem:products} and the above, a product of $n$ metric trees has Lipschitz dimension at most $n$. If each tree is non-trivial, the product contains an embedded copy of a cube in $\RR^n$, and hence has Lipschitz dimension equal to $n$.
\end{proof}

\subsection{Snowflakes of Euclidean spaces}\label{subsec:snowflake}
Unlike Nagata dimension (see \cite[Theorem 1.2]{LS}), Lipschitz dimension is not a quasisymmetric or even snowflake invariant, as we will discuss in Section \ref{sec:mappingproperties}. However, we have the following result:

\begin{theorem}\label{thm:snowflake}
For every $\epsilon\in (0,1]$, the Lipschitz dimension of the $\epsilon$-snowflake $X:=(\RR,|\cdot|^\epsilon)$ is $1$.
\end{theorem}

\begin{proof}
The case $\epsilon=1$ is immediate, so we assume $\epsilon < 1$.

By Assouad's theorem \cite[Theorem 12.2]{He}, there is a bi-Lipschitz embedding of $X$ into some Euclidean space. Let $n$ be the minimial dimension of a Euclidean space $\RR^n$ into which $X$ bi-Lipschitz embeds. Note that $n\geq 2$ since $\dim_H(X) > 1$. 

Let $Y\subseteq \RR^n$ be the image of $X$ under such a bi-Lipschitz embedding. Since Lipschitz dimension is clearly a bi-Lipschitz invariant, it suffices to show that $Y$ has Lipschitz dimension $1$.

Let $\pi:\RR^n\rightarrow\RR$ be the projection onto the first coordinate. We claim that $\pi|_Y$ is Lipschitz light.

Suppose not. Then, by Theorem \ref{prop:LLGH}, there is a weak tangent
$$ ((Z,z), (\RR,0),\hat{\pi})\in\WTan(\pi|_Y)$$
such that $\hat{\pi}$ is not light.

By Lemma \ref{lem:canonicaltangent}, the weak tangent package above may be viewed as a limit of rescalings inside $\RR^n$. In other words, there is an isometry $\iota$ from $Z$ onto a set $\hat{Y}\subseteq \RR^n$ with $\iota(z)=0$. Moreover, since rescalings and translations like those in Lemma \ref{lem:canonicaltangent} do not affect the linear map $\pi$, we have $\hat{\pi}=\pi\circ \iota$.

Thus, we have a weak tangent
$$ ( (\hat{Y}\subseteq \RR^n,0), (\RR,0), \pi) \in \WTan(\pi|_Y)$$
such that the linear projection $\pi$ is constant on a connected subset $E$ of $\hat{Y}$. 

Since $Y$ is bi-Lipschitz equivalent to $X=(\RR,|\cdot|^\epsilon)$, the space $\hat{Y}$ is also bi-Lipschitz equivalent to $(\RR, |\cdot|^\epsilon)$ by Lemma \ref{lem:GHmaps}. Therefore, any compact, connected subset $F$ of $E\subseteq \hat{Y}$ is bi-Lipschitz equivalent to $([-1,1],|\cdot|^\epsilon)$. 

$F$ is contained in $\pi^{-1}(p)$ (for some $p\in \RR$), which is isometric to $\RR^{n-1}$. Therefore, there is a bi-Lipschitz embedding
$$ h\colon ([-1,1],|\cdot|^\epsilon) \rightarrow \RR^{n-1}.$$
The mappings
$$ t\mapsto \lambda^\epsilon h(t/\lambda) \colon ([-\lambda,\lambda],|\cdot|^\epsilon) \rightarrow \RR^{n-1}$$
are then uniformly bi-Lipschitz, and so sub-converge to a bi-Lipschitz embedding of $X=(\RR, |\cdot|^\epsilon)$ into $\RR^{n-1}$ as $\lambda\rightarrow\infty$.

This contradicts our choice of $n$ as the minimal integer such that $X$ admits a bi-Lipschitz embedding into $\RR^n$. Thus, $\pi|_Y$ must in fact have been Lipschitz light, which makes $\dim_L(X) = \dim_L(Y) \leq 1$. Of course, $\dim_L(X) \geq 1$ by Observation \ref{obs:liptop}.
\end{proof}

\begin{corollary}\label{cor:Rnsnowflake}
The Cartesian product of $n$ snowflakes of $\RR$ has Lipschitz dimension $n$. In particular, each snowflake of $\RR^n$ has Lipschitz dimension $n$.
\end{corollary}
\begin{proof}
By Theorem \ref{thm:snowflake} and Lemma \ref{lem:products}, The Cartesian product of $n$ snowflakes of $\RR$ has Lipschitz dimension at most $n$. It has Lipschitz dimension at least its topological dimension, by Observation \ref{obs:liptop}, which is also $n$.

For the second statement, we simply observe that $(\RR^n,|\cdot|^\epsilon)$ is bi-Lipschitz equivalent to the Cartesian product of $n$ copies of $(\RR,|\cdot|^\epsilon)$.
\end{proof}

We can also prove a result about the Lipschitz dimensions of more general snowflakes in Euclidean space.
\begin{theorem}
Let $E\subseteq \RR^n$ be a closed set that is an $\alpha$-snowflake for some $\alpha\in (0,1)$. Let $k$ be an integer with
$$ k > n-\frac{1}{\alpha}.$$
Then
$$ \dim_L (E) \leq k.$$
\end{theorem}
\begin{proof}
Let $\pi:\RR^n \rightarrow \RR^k$ be an arbitrary choice of linear projection, which of course is $1$-Lipschitz. We claim that $\pi|_E$ is Lipschitz light.

Suppose not. Then, exactly as in Theorem \ref{thm:snowflake}, we find a set $\hat{E}\subseteq \RR^n$ such that
$$ ((\hat{E},0),(\RR^k,0),\pi) \in \WTan(\pi|_E),$$
and such that $\pi$ is constant on a non-trivial connected subset $C\subseteq\hat{E}$. Thus, $C$ is contained in some $\pi^{-1}(p)$, which is an $(n-k)$-plane in $\RR^n$.

On the other hand, $\hat{E}$ is also an $\alpha$-snowflake. Indeed, if $E$ is the image of a metric space $(Z,d)$ under an $\alpha$-snowflake map with constant $C$, then $\lambda E$ is the image of $(Z,\lambda^{1/\alpha} d)$ under an $\alpha$-snowflake map with constant $C$. It then follows from Proposition \ref{sublimit} and Lemma \ref{lem:GHmaps} that $\hat{E}$ is image of a metric space $\hat{Z}$ under an $\alpha$-snowflake map.

Since $\hat{E}$ is an $\alpha$-snowflake, each connected subset of $\hat{E}$ has Hausdorff dimension at least $\frac{1}{\alpha}$. (A connected set always has Hausdorff dimension at least $1$, and $\alpha$-snowflake maps multiply Hausdorff dimension by factor $\frac{1}{\alpha}$.)

Thus, using our assumption,
$$ \dim_H (C) = \frac{1}{\alpha} > n-k = \dim_H (\pi^{-1}(p)),$$
which is a contradiction to our observation that $C\subseteq \pi^{-1}(p)$.

Therefore, $\pi|_E$ must have been Lipschitz light, forcing $\dim_L (E) \leq k$.
\end{proof}

\subsection{Carnot groups}\label{subsec:carnot}
The so-called Carnot groups are central objects of study in the modern theory of analysis on metric spaces and non-smooth calculus. We begin this subsection with a very brief introduction to Carnot groups, referring the reader elsewhere for more background. We then show that non-abelian Carnot groups have infinite Lipschitz dimension, and follow that by discussing some consequences of this fact. We thank Bruce Kleiner for pointing out to us a number of years ago that Carnot groups should have infinite Lipschitz dimension.

\subsubsection{Background on Carnot groups}\label{subsub:carnotbackground}
We give a very brief background summary on Carnot groups. For more, we refer the reader to \cite{Mo,CDPT07, LD_primer}. Very little of the Lie group structure of Carnot groups is directly used in our arguments below, but it is necessary to set the stage.

A Carnot group is a simply connected nilpotent Lie group $\G$ whose Lie algebra $\mathfrak{g}$ admits a stratification
$$\mathfrak{g} = V_1 \oplus \cdots \oplus V_s,$$
where the first layer $V_1$ generates the rest via $V_{i+1} = [V_1,V_i]$ for all $1 \leq i \leq s$, and we set $V_{s+1} = \{0\}$.

Given an inner product $\langle\cdot,\cdot\rangle$ on the horizontal layer $V_1$, the associated sub-Riemannian Carnot--Carath\'eodory metric $d$ on $\G$ is defined by
$$ d(x,y) = \inf\left\{ \int_0^1 \langle \gamma'(t), \gamma'(t)\rangle ^{1/2} dt : \gamma \text{ horizontal curve joining } x \text{ to } y \right\},$$
where an absolutely continuous curve $\gamma \colon [0,1]\rightarrow\G$ is called \textit{horizontal} if $\gamma'(t) \in V_1$ for a.e. $t\in[0,1]$.

Like the standard Euclidean metric, which is just the special case in which the stratification has a single layer, the Carnot--Carath\'eodory distance $d$ is invariant under left-translations, the maps $L_x$ defined by $L_x(p)=x\cdot p$. It also admits a family of dilations: For each $\lambda>0$, there is a homeomorphism $\delta_\lambda\colon \G \rightarrow \G$ such that
$$ d(\delta_\lambda(x), \delta_\lambda(y)) = \lambda d(x,y) \text{ for each } x,y\in\G.$$
Together, these facts imply that every element of $\WTan(\G)$ is pointedly isometric to $(\G,0)$ itself.

The simplest non-abelian Carnot group is the (first) Heisenberg group $\mathbb{H}$. The underlying manifold of $\mathbb{H}$ is $\RR^3$, and its Lie algebra $\mathfrak{h}$ can be written
$$ \mathfrak{h} = V_1 \oplus V_2,$$
where $\dim(V_1)=2$, $\dim(V_2)=1$, $[V_1,V_1]=V_2$, and $[V_1,V_2]=0$. In exponential coordinates, $\mathbb{H}$ can be viewed as $\mathbb{C}\times\RR$ with the group law
$$ (z,t) \times (z',t') = \left( z+z', t+t'-\frac{1}{2}\text{Im}(z\overline{z'}) \right).$$
On the Heisenberg group, the Kor\'anyi metric
$$ d_K(p,q) = \|q^{-1} p\|,$$
where
$$ \|(z,t)\| = \left( |z|^4 + 16t^2 \right)^{1/4}$$
yields a bi-Lipschitz equivalent distance to $d$. (See \cite[p. 18]{CDPT07}.) If we define the standard projection $\pi:\mathbb{H}\rightarrow \mathbb{C}\cong\RR^2$ by $\pi(z,t)=z$, we see that $\pi$ is Lipschitz and that $\pi^{-1}(y)$ is a snowflake of $\RR$ for each $y\in\RR$.

The main result we will use below about Carnot groups is the celebrated Pansu differentiation theorem:
\begin{theorem}[Pansu \cite{Pa89}]\label{thm:Pansu}
Let $f \colon \G_1 \rightarrow \G_2$ be a Lipschitz map between Carnot groups. Then for almost every $x \in \G_1$, the sequence of maps
$$\delta_{\lambda} \circ (L_{f(x)^{-1}} \circ f \circ L_x) \circ \delta_{\lambda^{-1}}$$
converges uniformly on compact sets, as $\lambda \rightarrow \infty$, to a Lie group homomorphism $Df(x) \colon \G_1 \rightarrow \G_2$ that commutes with the dilations.
\end{theorem}
We will use Theorem \ref{thm:Pansu} below in the case where $\G_1$ is non-abelian and $\G_2$ is a Euclidean space $\RR^n$. In that case, we note that $Df(x)$ must collapse the (connected) commutator subgroup of $\G_1$. Note also that, in the setting of Theorem \ref{thm:Pansu},
$$ ((\G_1,0),(\G_2,0), Df(x))\in \Tan(f,x) \subseteq \WTan(f).$$

Carnot groups are doubling metric spaces. (See \cite[p. 1]{LD_primer} and note that Ahlfors regular spaces are always doubling.), Therefore, they have finite Hausdorff and Assouad dimension. In addition, their Nagata dimensions are equal to their topological dimensions \cite[Theorem 4.2]{LDR}, and hence also finite, though generally smaller.

Nonetheless, in the next section, we show that non-abelian Carnot groups have infinite Lipschitz dimension.

\subsubsection{Lipschitz dimension of Carnot groups}
\begin{theorem}\label{thm:Carnot}
If $\G$ is a non-abelian Carnot group, then $\dim_L (\G) = \infty$.
\end{theorem}
\begin{proof}
Suppose there was a Lipschitz light map $f\colon \G\rightarrow \RR^n$. Then by Theorem \ref{thm:Pansu}, and the remark following, there is
$$ ((\G,0),(\RR^n,0), Df(x))\in \WTan(f)$$
such that $Df(x)$ is a group homomorphism that commutes with dilations. In particular, $Df(x)$ must collapse the (connected) commutator subgroup of $\G$ to a point.

However, the mapping $Df(x)$ is Lipschitz light by Corollary \ref{cor:LLtan}, so cannot collapse a connected set to a point. This is a contradiction.
\end{proof}

We note that the same result holds for positive-measure subsets of Carnot groups:
\begin{corollary}\label{cor:carnotsubset}
Let $\G$ be a non-abelian Carnot group and let $K\subseteq \G$ be compact with positive measure. Then $\dim_L (K) = \infty$.
\end{corollary}
\begin{proof}
By \cite[Proposition 3.1]{LDtan}, there is a point $x\in K$ and a tangent $(\hat{K},\hat{x})\in \Tan(K,x)$ such that $\hat{K}$ is isometric to $\G$. It follows from Corollary \ref{cor:Ldimtan} that 
$$ \dim_L (K) \geq \dim_L (\hat{K}) = \dim_L (\G) = \infty.$$
\end{proof}

\subsubsection{Quasi-isometric non-embedding for Carnot groups}
As a corollary of Theorem \ref{thm:Carnot}, we prove a ``coarse'' non-embedding result for Carnot groups, Corollary \ref{cor:carnotcoarse}. Our theorem overlaps with a result of Pauls \cite{Pauls}, but our approach is somewhat different.

We recall a notion from coarse geometry: A \textit{quasi-isometric embedding} of a space $X$ into a space $Y$ is a (not necessarily continuous) map $g:X\rightarrow Y$ with constants $C\geq 1$ and $\epsilon>0$ such that
$$ C^{-1}d(x,x') -\epsilon \leq d(g(x),g(x')) \leq Cd(x,x') + \epsilon$$
for all $x,x'\in X$. Quasi-isometric embeddings are coarse generalizations of bi-Lipschitz embeddings.

Our methods give a short proof of the following result.
\begin{corollary}\label{cor:carnotcoarse}
If $\G$ is a non-abelian Carnot group, then $\G$ does not admit a quasi-isometric embedding into any space of finite Lipschitz dimension. In particular, $\G$ does not admit a quasi-isometric embedding into any finite product of trees or finite-rank Euclidean building.
\end{corollary}

The statement about trees and buildings in Corollary \ref{cor:carnotcoarse} already follows from a general result of Pauls \cite[Theorem C]{Pauls}. Both approaches rely at heart on Pansu's theorem. One advantage of our short approach is that it does not require proving a ``metric differentiation'' form of Pansu's theorem (see \cite[Theorem 4.7]{Pauls}) but rather relies directly on the original result of Pansu. On the other hand, Pauls' result allows for quite general targets, including infinite-dimensional spaces, which our result does not address. 

\begin{proof}[Proof of Corollary \ref{cor:carnotcoarse}]
Suppose $\G$ admits a quasi-isometric embedding $g\colon \G \rightarrow Y$, where $Y$ has finite Lipschitz dimension. There are constants $C\geq 1$ and $\epsilon>0$ such that
$$ C^{-1}d(x,x') -\epsilon \leq d(g(x),g(x')) \leq Cd(x,x') + \epsilon$$
for all $x,x'\in \G$.

Let $N$ be a $2C\epsilon$-net in $\G$ containing $0$. On the one hand, $g|_N$ is easily seen to be a bi-Lipschitz embedding of $N$ into $Y$, and therefore $N$ has finite Lipschitz dimension.

On the other hand, the pointed spaces $\left(\delta_{1/k}(N),0\right)$ converge to the pointed space $(\G,0)\in \WTan(N)$ as $k\in\mathbb{N}$ tends to infinity. It follows from Corollary \ref{cor:Ldimtan} that
$$ \dim_L(\G) \leq \dim_L(N) < \infty,$$
which contradicts Theorem \ref{thm:Carnot}. Therefore, there can be no such quasi-isometric embedding $g$.

The statement about trees and buildings now follows from Corollary \ref{cor:treebuilding}.
\end{proof}

\subsubsection{Carnot groups as counterexamples}

We close this discussion of Carnot groups by observing that they, in particular the first Heisenberg group $\mathbb{H}$, provide counterexamples to two natural hopes for Lipschitz dimension. 

First of all, in contrast to Proposition \ref{lem:union}, we observe that the finiteness of Lipschitz dimension is not stable under countable unions, even locally finite ones. Indeed, consider a $1$-net $N$ in the Heisenberg group $\bH$, with $0\in N$. Exactly as in the proof of Corollary \ref{cor:carnotcoarse}, we must have $\dim_L(N)=\infty$, even though $N$ is countable.

Next, we observe that the Heisenberg group also serves as a counterexample to any ``Hurewicz-type'' theorem for Lipschitz dimension. Recall first the classical Hurewicz theorem for topological dimension, which we state in the compact case: If $f:X\rightarrow Y$ is a continuous map between compact metric spaces, then 
$$ \dim_T (X) \leq \dim_T (Y) + \sup\{\dim_T \left(f^{-1}(y)\right): y\in Y\}.$$
See, for example, \cite[Theorem III.6]{Nagata}.

No such result holds with $\dim_L$ replacing $\dim_T$: Let $X$ denote the closed unit ball in the Heisenberg group $\mathbb{H}$, let $\pi:X\rightarrow Y:=\RR^2$ denote the restriction to $X$ of the standard projection from $\mathbb{H}$ to $\RR^2$. Then $\dim_T Y = 2$. Furthermore, for each $y\in \RR^2$, $\pi^{-1}(y)$ is contained in a space that is bi-Lipschitz equivalent to a $\frac{1}{2}$-snowflake of $\RR$, and hence has Lipschitz dimension $\leq 1$ by Theorem \ref{thm:snowflake}. However, $X$ itself has infinite Lipschitz dimension.

\subsection{Subsets of $\RR$}
In this subsection, we characterize the Lipschitz dimension of subsets of $\RR$ by a simple metric property.

\begin{definition}
A set $E$ in a metric space $X$ is called \textit{porous}, with constant $c>0$, if for every $x\in E$ and $r>0$, there is a point $y$ with
$$ B(y,cr) \subseteq B(x,r) \setminus E.$$
\end{definition}

\begin{proposition}\label{prop:porous}
Let $E\subseteq \RR$. Then the following are equivalent
\begin{enumerate}[(i)]
\item $E$ has Lipschitz dimension $0$.
\item Every weak tangent of $E$ is totally disconnected.
\item $E$ is porous.
\end{enumerate}
\end{proposition}
Statements analogous to the equivalence (ii) $\Leftrightarrow$ (iii) hold in all $\RR^n$ and are standard.
\begin{proof}[Proof of Proposition \ref{prop:porous}]
First, assume $E$ has Lipschitz dimension $0$.  Then every weak tangent of $E$ has Lipschitz dimension $0$, by Corollary \ref{cor:Ldimtan}, hence topological dimension $0$, hence is totally disconnected. Thus, (i) implies (ii).

Suppose $E$ satisfied condition (ii) but $E$ was not porous. Then we could find balls $B(x_i, r_i)\subseteq \RR$, with $x_i\in E$, such that
$$ N_{1/i}(\overline{B}(x_i,r_i) \cap E) \supseteq \overline{B}(x_i,r_i).$$
Passing to a weak tangent of $E$ along the sequence of scales $\lambda_i=1/r_i$ and the sequence of points $x_i$, we obtain a weak tangent $(\hat{E},0)\in\WTan(E)$ such that $\hat{E}$ contains an isometric copy of $[-1,1]$. This contradicts (ii), proving that (ii) implies (iii). 

Finally, suppose $E\subseteq \RR$ is porous. Then no weak tangent of $E$ contains a non-trivial interval. Indeed, suppose $\{(\lambda_k E, x_k)\}$ converged to a pointed metric space $(\hat{E},0)$ containing a non-trivial interval. Then, along a subsequence, the sets $\{\lambda_k (E - x_k)\}$ would converge in $\RR$ to a set $F$ that is isometric to $\hat{E}$, and so contains a non-trivial interval. It would follow that, for arbitrarily large $\lambda$, there is an interval $I_\lambda$ such that $E\cap I_\lambda$ is $\frac{1}{\lambda}$-dense in $I_\lambda$, which contradicts the porosity of $E$. 

Thus, every weak tangent mapping of the constant map $\kappa:E\rightarrow\RR^0$ is light, and hence $\kappa:E\rightarrow \RR^0$ is Lipschitz light. This proves that (iii) implies (i).
\end{proof}

For sets in $\mathbb{R}^n$, we do not know whether having Lipschitz dimension $\leq n-1$ is equivalent to porosity.

One direction is clear: If a set in $\RR^n$ has Lipschitz dimension $\leq n-1$, it must be porous. If it is not, then by an argument similar to that in Proposition \ref{prop:porous}, it has a weak tangent containing an isometric copy of a ball in $\RR^n$, contradicting Corollary \ref{cor:Ldimtan}.

\begin{question}\label{q:lipdimporous}
Is it true that a set $E\subseteq \RR^n$ has Lipschitz dimension $\leq n-1$ if and only if it is porous?
\end{question}

\begin{remark}
The answer to Question \ref{q:lipdimporous} is ``yes'' if one replaces Lipschitz dimension by Nagata dimension. One direction (Nagata dimension $\leq n-1$ implies porosity) follows from the same argument as in the remark above Question \ref{q:lipdimporous}, since Nagata dimension can also only drop under weak tangents \cite[Proposition 2.18]{LDR}. For the other direction, it is well-known that porous subsets of $\RR^n$ have Assouad dimension $<n$ and hence Nagata dimension $\leq n-1$ by \cite[Theorem 1.1]{LDR}.
\end{remark}

\subsection{Self-covering sets and classical fractals}\label{subsec:fractal}
In this subsection, our goal is to show that some classical fractals in the plane have Lipschitz dimension $1$. As concrete examples, our results apply to the standard Sierpi\'nski carpets $S_p$, indexed by odd integers $p\geq 3$. Recall that for such $p$, the ``first generation'' $S^1_p\subseteq \RR^2$ is formed by dividing the unit square into axis-parallel subsquares of side length $\frac{1}{p}$ in the usual way and removing the central square. The $n$th generation $S^n_p$ is formed by doing the same procedure on each of the squares of side length $p^{-(n-1)}$ remaining in $S^{n-1}_p$. The Sierpi\'nski carpet $S_p$ is defined as $\cap_{n=1}^\infty S^n_p$.

Our results will also apply to the standard Sierpi\'nski gasket, similarly formed by starting with an equilateral triangle in the plane, dividing it into four congruent equilateral triangles, removing the central triangle, and then iterating this construction on the remaining three triangles of half the size. See, for example, \cite[p. 7-8]{DS97} for pictures and descriptions of these constructions.

We frame our result for a certain class of sets that includes the above examples, which we now describe. For a compact set $K$, a \textit{rescaled translate} of $K$ is a set of the form $sK+v$ for some $s>0$ and $v\in\mathbb{R}^n$. 
\begin{definition}\label{def:sc}
We call a compact set $K\subseteq \RR^n$ \textit{self-covering} if there are constants $N\in\mathbb{N}$ and $C>0$ such that the following holds: For each $x\in K$ and $r>0$, there are rescaled translate $K_1, \dots, K_M$ of $K,$ inside $K$, such that 
\begin{itemize}
\item $M\leq N$,
\item $\diam K_i \leq Cr$, and
\item $\overline{B}(x,r)\cap K \subseteq \cup K_i$.
\end{itemize}
\end{definition}
In other words, a set $K$ is self-covering if every ball of radius $r$ in $K$ can be covered by a controlled number of rescaled copies of $K$ of size at most $Cr$. Note that we allow the copies of $K$ covering $B(x,r)\cap K$ to overlap arbitrarily and to contain points outside of $K$, but we do not allow rotations.

It is easy to see that the Sierpi\'nski carpets $S_p$ and the Sierpi\'nski gasket are self-covering. On the other hand, the self-covering property is somewhat different than standard notions of self-similarity; for example, the set $[0,1] \cup [2,3]$ in $\RR$ is self-covering. For non-examples, we point out that the set $\{0\}\cup\{1, \frac{1}{2}, \frac{1}{3}, \dots\}$ in $\RR$ and the unit circle $\mathbb{S}^1$ in $\RR^2$ are examples of non-self-covering sets.

Of course, the whole unit cube in $\RR^n$ is also an example of a self-covering set, so we cannot expect to interestingly bound the Lipschitz dimension based only on the self-covering property. Our additional assumption is that the self-covering set does not contain any non-trivial line segments in some fixed direction.

\begin{theorem}\label{theorem:ss}
Let $K\subseteq\RR^n$ be compact and self-covering, according to Definition \ref{def:sc}. Assume that there exists $v\in\mathbb{S}^{n-1}$ such that $K$ contains no non-trivial line segment in direction $v$. Then the Lipschitz dimension of $K$ is at most $n-1$.
\end{theorem}

\begin{proof}[Proof of Theorem \ref{theorem:ss}]
Let $K$ and $v$ be as in the theorem. Assume without loss of generality that $\diam(K) = 1$ and $0\in K$.

Let $\pi$ denote the orthogonal projection from $\RR^n$ onto the orthogonal complement $V$ of $v$; of course, $V$ is isometric to $\RR^{n-1}$. 

We claim that $\pi|_K$ is Lipschitz light, which will suffice to prove the theorem. The spirit of this argument is similar to some above that use Gromov-Hausdorff convergence. However, in this setting we need to be a bit careful \textit{not} to identify isometric sets, as we want to avoid rotation.

Suppose that $\pi|_K$ is not Lipschitz light. Then for each $j\in\mathbb{N}$ there is a set $W_j\subseteq V$ of diameter at most $r_j$ such that $\pi^{-1}(W_j)$ contains an $r_j$-path $P_j$ with $R_j:=\diam(P_j) \geq jr_j$.

Let $B_j=\overline{B}(x_j,R_j)\cap K$ be a ball in $K$ of radius $R_j$ containing $P_j$, where $x_j$ is the initial point of $P_j$. By Definition \ref{def:sc}, there are rescaled translates $K_j^1,\dots, K_j^{N_j}$ of diameter at most $CR_j$ covering $B_j$, with $N_j\leq N$. Note that we may freely assume that each $K_j^i$ actually intersects $B_j$, and therefore is contained in $\overline{B}(x_j, (C+1)R_j)$.

By passing to a subsequence, we may also assume that that there is $M\in \{1,2,\dots, N\}$ such that $N_j=M$ for all $j\in\mathbb{N}$.

For each $i\in\{1,\dots,M\}$, consider the sequence of sets
$$ \frac{1}{R_j}(K_j^i - x_j).$$
This is a sequence of rescaled translates of $K$, all contained in $\overline{B}(0,C+1)$. Thus, we may pass to a subsequence (which we continue to label by $j$) such that for each $i$, this sequence converges in the Hausdorff sense (equivalently, in the sense of Definition \ref{sets}) to a set $K^i_\infty$ that is a rescaled translate of $K$. Indeed, each set $\frac{1}{R_j}(K_j^i - x_j)$ is simply $s^i_j K^i_j + v^i_j$ for some $s^i_j\in[0, 2(C+1)]$ and $v^i_j\in \overline{B}(0,C+1)$, so we may simply pass to a subsequence along which $\{s^i_j\}_{j=1}^\infty$ and $\{v^i_j\}_{j=1}^\infty$ both converge.

In particular, our assumption on $K$ implies that no $K^i_\infty$ can contain a non-trivial line segment in direction $v$.

Let $K_\infty = \cup_{i=1}^M K^i_\infty$. 

By passing to a further subsequence, we may assume also that the sets
$$ \frac{1}{R_j}(P_j - x_j)$$
converge to a compact subset $P_\infty\subseteq \RR^n$. We also claim that $P_\infty\subseteq K_\infty$: If $y\in P_\infty$, then $y=\lim y_j$ for some $y_j \in \frac{1}{R_j}(P_j - y_j)$. Each such $y_j$ is in some $\frac{1}{R_j}(K_j^i - x_j)$. Therefore there is some $i_0\in\{1,\dots,M\}$ such that $y_j \in \frac{1}{R_j}(K_j^{i_0} - x_j)$ for infinitely many $j$, from which it follows that $y\in K^i_\infty\subseteq K_\infty$.

By Lemma \ref{lem:pathconvergence}, $P_\infty$ is a connected set, and it has diameter $1$. Moreover,
$$ \diam(\pi(P_\infty)) = \lim_{j\rightarrow\infty} \diam\left( \pi\left(\frac{1}{R_j}(P_j - x_j)\right) \right) \leq \lim_{j\rightarrow \infty} \frac{1}{j} =0.$$

Thus, $P_\infty \subseteq K_\infty$ is a non-trivial line segment in direction $v$.

To conclude the proof, we now argue that in fact some $K^i_\infty$ must contain a non-trivial sub-segment of $P_\infty$. Indeed, if not, then $K^i_\infty \cap P_\infty$ has empty interior in $P_\infty$ for each $j$. However, by the Baire Category Theorem, $P_\infty$ cannot be the union of a finite collection of subsets with empty interior. Thus, some $K^i_\infty \cap P_\infty$ must contain a non-trivial sub-segment of $P_\infty$. Since $K^i_\infty$ is a rescaled translate of $K$, this is a contradiction.
\end{proof}

\begin{corollary}
For each odd $p\in\mathbb{N}$, the Sierpi\'nski carpets $S_p$ have Lipschitz dimension $1$. The same holds for the Sierpi\'nski gasket $G$.
\end{corollary}
\begin{proof}
The Sierpi\'nski carpets $S_p$ and Sierpi\'nski gasket $G$ are easily seen to satisfy Definition \ref{def:sc}. Moreover, $S_p$ contains no non-trivial line segments in directions of irrational slope (see \cite[Corollary 4.5]{DurandTyson} or \cite[Theorem 3.4]{ChenNiemeyer}), and the gasket $G$ clearly contains no nontrivial vertical line segments. Thus, these fractals all satisfy the conditions of Theorem \ref{theorem:ss} and so have Lipschitz dimension at most $1$. Since each contains some non-trivial line segments, their Lipschitz dimensions must be equal to $1$.
\end{proof}

\section{Relationship to other dimensions}\label{sec:relation}
We have already seen in Observation \ref{obs:liptop} that the Lipschitz dimension of $\sigma$-compact metric spaces is bounded below by topological dimension. Here we briefly indicate the relation (or lack thereof) between Lipschitz dimension and the Nagata, Hausdorff, and Assouad dimensions. These three dimensions were defined in subsection \ref{subsec:dimensions}.

\subsection{Nagata dimension}
As we recalled in subsection \ref{subsec:dimensions}, $\dim_N(X) \geq \dim_T(X)$ for every metric space $X$. We show that Lipschitz dimension provides an upper bound for Nagata dimension.

\begin{lemma}
If $f:X\rightarrow Y$ is Lipschitz light, then $\dim_N (X) \leq \dim_N (Y)$.
\end{lemma}
\begin{proof}
Without loss of generality, we may assume that $f$ is $1$-Lipschitz and that $\dim_N (Y) = n < \infty$. Fix $s$ and consider a $cs$-bounded cover $\{B_i\}$ of $Y$ with $s$-multiplicity at most $n+1$. We may also assume without loss of generality that $c\geq 1$. 

For each $i$, let $\{U^i_j\}$ denote the $cs$-components of $f^{-1}(B_i)$. Then, because $f$ is Lipschitz light with some constant $C\geq 1$, we have
$$ \diam(U^i_j) \leq Ccs $$
for all $i,j$.

We claim that $\{U^i_j\}_{i,j}$ form a cover of $X$ with $s$-multiplicity at most $n+1$. Consider any set $E\subset X$ with $\diam(E)\leq s$. First of all, note that for each fixed $i$, $E$ can meet $U^i_j$ for at most one value of $j$. Indeed, if $E$ met both $U^i_j$ and $U^i_{k}$, then there would be $x\in U^i_j$ and $y\in U^i_k$ with $d(x,y)\leq s \leq cs$, in which case $U^i_j$ and $U^i_k$ would be the same $cs$-component, i.e., we would have $j=k$.

So we must show that $E$ meets some $U^i_j$ for at most $n+1$ values of $i$. This is the same as saying that $f(E)$ meets $B_i$ for at most $n+1$ values of $i$. This is in fact the case, because $\diam f(E) \leq s$, as $f$ is $1$-Lipschitz, and because $\{B_i\}$ has $s$-multiplicity at most $n+1$. This completes the proof.
\end{proof}

\begin{corollary}\label{cor:naglipdim}
For any metric space $X$, $\dim_N X \leq \dim_L X$.
\end{corollary}
\begin{proof}
This follows immediately from the previous lemma and the fact that the Nagata dimension of $\mathbb{R}^n$ is $n$.
\end{proof}

On the other hand, Nagata dimension provides no non-trivial upper bound for Lipschitz dimension, as demonstrated by Theorem \ref{thm:Carnot} and \cite[Theorem 4.2]{LDR}, which together say that Carnot groups have infinite Lipschitz dimension and finite Nagata dimension.

Nagata dimension and Lipschitz dimension do agree for $0$-dimensional spaces:

\begin{proposition}\label{prop:zerodim}
A metric space $X$ has Lipschitz dimension $0$ if and only if it has Nagata dimension $0$.
\end{proposition}
\begin{proof}
By Corollary \ref{cor:naglipdim}, we always have
$$ \dim_L(X) \geq \dim_N(X),$$
and so if $\dim_L(X)=0$ then $\dim_N(X)=0$.

Conversely, suppose the Nagata dimension of $X$ is zero. That means that, for every $s>0$, there is a $cs$-bounded cover of $X$ with $s$-multiplicity at most $1$.

Let $f\colon X\rightarrow \RR^0$ be the constant map. We claim that $f$ is Lipschitz light. This just means that for every $s>0$, the $s$-components of $X$ have diameter at most $cs$. Consider the cover of $X$ given by the Nagata dimension in the previous paragraph. Any $s$-component of $X$ must be contained in a single set in the cover, so has diameter at most $cs$. So $f$ is Lipschitz light. 
\end{proof}

Interestingly, the author does not know an example of a space with Nagata dimension $1$ and Lipschitz dimension greater than $1$.

\begin{question}\label{q:nagata1}
Is there a compact metric space with Nagata dimension $1$ and Lipschitz dimension greater than $1$?
\end{question}
This question is interesting in light of Theorems \ref{theorem:CKgraph} and \ref{theorem:CKL1}.

\subsection{Hausdorff and Assouad dimension}\label{subsec:HausdorffAssouad}
There is in general no relationship between the Lipschitz dimension and the Hausdorff or Assouad dimension of a space. The following two propositions indicate this.

Building on a construction of Laakso \cite{La00}, Cheeger and Kleiner \cite{CK13_PI} give a very flexible construction of metric spaces with Lipschitz dimension $1$, including examples with arbitrary Hausdorff and Assouad dimensions. 
\begin{proposition}[\cite{CK13_PI}]\label{prop:Laakso}
For every $\alpha \geq 1$, there is a compact metric space of Lipschitz dimension $1$ and Hausdorff and Assouad dimension equal to $\alpha$.
\end{proposition}
This shows that the Hausdorff and Assouad dimensions can be larger than Lipschitz dimension by any desired amount.

The reverse situation can also happen: the Lipschitz dimension may be any amount larger than the Hausdorff and/or Assouad dimensions. Indeed, as noted in subsection \ref{subsub:carnotbackground}, Carnot groups have finite Hausdorff and Assouad dimensions, but have infinite Lipschitz dimension by Theorem \ref{thm:Carnot}.

\section{Cheeger's analytic dimension}\label{sec:Cheeger}
In this section, we describe Cheeger's version of Rademacher's theorem on certain non-smooth metric measure spaces, which equips them with a type of ``measurable cotangent bundle'', and we show that Lipschitz dimension bounds the dimension of this cotangent bundle.

\subsection{Cheeger's differentiation theory and Lipschitz quotient mappings}

Recall that Rademacher's theorem says that a Lipschitz mapping from $\RR^n$ to $\RR$ is differentiable almost everywhere, with respect to Lebesgue measure. In \cite{Ch99}, Cheeger gave a far-reaching generalization of this result to a large class of non-smooth metric measure spaces. In order to do so, he defined a very general notion of differentiable structure on a metric measure space. (The name ``Lipschitz differentiability space'' used below for this notion was coined by Bate \cite{Bate}.)

\begin{definition}[\cite{Ch99} ]\label{def:LDspace}
A metric measure space $(X,\mu)$ is called a \textit{Lipschitz differentiability space} if it satisfies the following condition. There are countably many Borel sets (``charts'') $U_i$ of positive measure covering $X$, positive integers $n_i$ (the ``dimensions of the charts''), and Lipschitz maps $\phi_i\colon X\rightarrow\mathbb{R}^{n_i}$ with respect to which any Lipschitz function $f \colon X \rightarrow \RR$ is differentiable almost everywhere, in the sense that for each $i$ and for $\mu$-almost every $x\in U_i$, there exists a unique $df(x) \in\mathbb{R}^{n_i}$ such that
\begin{equation} \label{LD}
\lim_{y\rightarrow x} \frac{|f(y) - f(x) -  df(x) \cdot(\phi_i(y)-\phi_i(x))|}{d(x,y)} = 0.
\end{equation}
Here $df(x) \cdot(\phi_i(y)-\phi_i(x))$ denotes the standard scalar product in $\mathbb{R}^{n_i}$.
\end{definition}

Although the choice of charts $(U_i, \phi_i)$ is by no means uniquely determined, the numbers $n_i$ are, in the sense that if $(U, \phi)$ and $(V,\psi)$ are charts and $\mu(U \cap V)>0$, then their associated dimensions must be the same. Thus, the numbers $n_i$ reflect something about the geometry of the space $(X,d,\mu)$, which motivates the following, chart-independent, definition:
\begin{definition}\label{def:analyticdim}
If $(X,d,\mu)$ is a Lipschitz differentiability space, we call the supremum of the numbers $n_i$  from Definition \ref{def:LDspace} the \textit{analytic dimension} of $X$, and denote it $\dim_C(X)$.
\end{definition}

For a nice introduction to Cheeger's theory, we refer the reader to \cite{KM} and for more recent developments in the subject to \cite{Bate, Schioppa, EB}. For more specific results on the interaction between analytic dimension and metric geometry, which is an active area of research, we refer the reader to \cite{GCD15, BL, GCDK, KS}.

The main theorem of \cite{Ch99} is that all the so-called ``Poincar\'e inequality (PI) spaces'' are Lipschitz differentiability spaces. Examples of these include Euclidean spaces and Carnot groups \cite{He}, as well as a selection of more exotic examples appearing in \cite{BP, La00, CK13_PI, KS}. The full spectrum of possibilities seems to not yet be well-understood.

A key tool in the study of Lipschitz differentiability spaces has been the following notion.
\begin{definition}[\cite{BJLPS}]\label{LQdef}
Let $X$ and $Y$ be metric spaces and $f \colon X\rightarrow Y$ a mapping. We say that $f$ is a \textit{Lipschitz quotient mapping} if there are constants $C,c>0$ such that
\begin{equation}\label{LQdefeqn}
B(f(x),cr) \subseteq f(B(x,r)) \subseteq B(f(x),Cr)
\end{equation}
for all $x\in X$ and $r>0$.
\end{definition}
Note that the second inclusion in \eqref{LQdefeqn} simply says that the mapping is $C$-Lipschitz. Lipschitz quotient mappings were first defined and studied in \cite{BJLPS,JLPS}, where the following path-lifting property was observed. (For a proof in the generality below, see \cite[Lemma 3.3]{GCDK}.)

\begin{lemma}\label{lem:pathlift}
Let $X$ be a proper metric space and let $f\colon X\rightarrow Y$ be a Lipschitz quotient map. Let $\gamma\colon [0,T]\rightarrow Y$ be a $1$-Lipschitz curve with $\gamma(0)=f(x)$. Then there is a Lipschitz curve $\tilde{\gamma}\colon [0,T]\rightarrow X$ with $\tilde{\gamma}(0)=x$ such that $ f\circ \tilde{\gamma} = \gamma$.
\end{lemma}

Lipschitz quotient maps enter the study of Lipschitz differentiability spaces through the following proposition. Independent proofs of this fact were given in \cite[Theorem 5.56]{Schioppa} and (in the doubling case) \cite[Corollary 5.1]{GCD15}. A stronger statement appears in \cite[Theorem 1.11]{CKS} but is not needed here.
\begin{proposition}\label{prop:LQblowup}
Let $(X,d,\mu)$ be a complete Lipschitz differentiability space with a chart $(U,\phi\colon X\rightarrow \RR^k)$. Then for almost every $x\in X$ and every mapping package
$$ \left((\hat{X},\hat{x}), (\RR^k,0), \hat{\phi}\right) \in \Tan(\phi,x),$$
the mapping $\hat{\phi}$ is a Lipschitz quotient map of $\hat{X}$ onto $\RR^k$.
\end{proposition}
If $(X,d)$ is not a metrically doubling metric space, interpreting $\Tan(\phi,x)$ requires a bit of care. However, this can be done in general similarly to how we handle it in the proof of Theorem \ref{thm:Cheegerbound} below. See also \cite[Remark 2.11]{GCDK}.

\subsection{Lipschitz dimension bounds analytic dimension}
It was proven in \cite[Corollary 5.99]{Schioppa} and \cite[Corollary 8.5]{GCD15} that Assouad dimension is an upper bound for the analytic dimension of Lipschitz differentiability spaces. In fact, a stronger result is now known to hold: Hausdorff dimension is an upper bound for analytic dimension. This follows from  \cite[Theorem 1.1]{DMR}; see also the approaches in \cite{KM} and \cite{GP}.

On the other hand, it is a very interesting open question whether Nagata dimension bounds analytic dimension: see \cite[Question 1.2]{KS}.

We show here that Lipschitz dimension is an upper bound for analytic dimension. Note that by the results in subsection \ref{subsec:HausdorffAssouad}, this neither implies nor is implied by the above-mentioned results concerning Assouad and Hausdorff dimension.

\begin{theorem}\label{thm:Cheegerbound}
Let $(X,d,\mu)$ be a complete Lipschitz differentiability space. Then $\dim_C X \leq \dim_L X$.
\end{theorem}
\begin{proof}
Without loss of generality, we may assume that $n:= \dim_L X <\infty$, otherwise the theorem is trivial. Let $f\colon X\rightarrow\RR^n$ be a Lipschitz light map.

Let $k$ denote the analytic dimension of $(X,d,\mu)$, so that there is a chart 
$$(U,\phi\colon X\rightarrow \RR^k)$$
in $X$. Our goal is to show that $k\leq n$, so assume to the contrary that $k>n$. 

Lipschitz differentiability spaces satisfy a property called ``pointwise doubling'', which in particular implies that they can be covered up to measure zero by compact, metrically doubling subsets. (See \cite[Corollary 2.6]{BS} and \cite[Lemma 8.3]{Bate}.) We can therefore find a compact, metrically doubling subset $A\subset U$ with $\mu(A)>0$. Moreover, by \cite[Corollary 2.7]{BS}, the space $(A,d,\mu)$ is a complete Lipschitz differentiability space consisting of one chart $(A,\phi\colon A\rightarrow \RR^k)$.

We now work entirely on $A$ and forget about the rest of $X$. Of course, $f$ restricts to a Lipschitz light map $f|_A\colon A \rightarrow \RR^n$, which we continue to call $f$.

Choose a point $x\in A$ at which each of the $n$ $\RR$-valued component functions $f_i$ of $f$ are differentiable. We may also choose $x$ such that the conclusion of Proposition \ref{prop:LQblowup} holds at $x$. By rescaling and passing to a suitable subsequence, we find
$$ (\hat{A},\hat{x}) \in \Tan(A,x),$$
$$ \left( (\hat{A},\hat{x}), (\RR^n,0), \hat{f} \right) \in \Tan(f,x),$$
and
$$\left( (\hat{A},\hat{x}), (\RR^k,0), \hat{\phi} \right) \in \Tan(\phi,x).$$
From the definition of differentiability in \eqref{LD}, there is be a linear map $L:\RR^k \rightarrow \RR^n$, formed from the $df_i$, such that
$$ \hat{f} = L \circ \hat{\phi}.$$

Since we have assumed that $k>n$, $L$ has a non-trivial kernel. In other words, there is a line $\ell\in\RR^k$ such that $L(\ell) = \{0\}$.

By Proposition \ref{prop:LQblowup}, $\hat{\phi}$ is a Lipschitz quotient map. By Lemma \ref{lem:pathlift}, there must therefore be a non-trivial curve $\gamma\subseteq \hat{A}$ such that $\hat{\phi}(\gamma)\subseteq L$. It follows that
$$\hat{f}(\gamma) =  L \circ \hat{\phi}(\gamma) = \{0\},$$
i.e., that $\hat{f}$ is constant on $\gamma$.

On the other hand, $\hat{f}$ is a light mapping, by Proposition \ref{prop:LLlimit}. It can therefore not collapse the non-trivial curve $\gamma$ to a point. This is a contradiction, and therefore we must have $k\leq n$.
\end{proof}

\section{Mapping properties}\label{sec:mappingproperties}
In this section, we discuss the invariance and non-invariance properties of Lipschitz dimension under various classes of mappings. We show that Lipschitz light mappings cannot decrease but can arbitrarily increase Lipschitz dimension, and point to examples that show that Lipschitz dimension is in general not a quasisymmetric or snowflake invariant. We conclude by studying David--Semmes regular mappings and using them to prove Corollary \ref{cor:DS}, which provides a necessary condition for certain metric spaces to admit non-degenerate Lipschitz maps between them.

\subsection{Lipschitz light mappings}\label{subsec:LLmapping}
In subsection \ref{subsec:treesbuildings}, we already made the easy observation that if $f\colon X\rightarrow Y$ is Lipschitz light, then $\dim_L (X) \leq \dim_L (Y)$. In other words, Lipschitz light mappings can only raise Lipschitz dimension.

Here, we observe that this inequality may be strict (even if $f$ is surjective). The preliminary lemma we need is the following:

\begin{lemma}
Let $X$ be a metric space of Nagata dimension $0$, let $Y$ be a metric space, and let $f:X\rightarrow Y$ be Lipschitz. Then $f$ is Lipschitz light.
\end{lemma}
\begin{proof}
We showed in Proposition \ref{prop:zerodim} that $X$ must also have Lipschitz dimension $0$, i.e., that $X$ admits a Lipschitz light map to to the one-point metric space $\RR^0$. It follows that there is a constant $C>0$ such that every $r$-path $P$ in $X$ has diameter at most $Cr$. Hence, if $W\subseteq Y$ has $\diam(W)\leq r$, then every $r$-component of $f^{-1}(W)$ has diameter at most $Cr$, making $f$ Lipschitz light.
\end{proof}

The following fact is probably well-known, but we include a proof. It is an analog of the well-known topological fact that every compact metric space is a continuous image of the Cantor set.
\begin{proposition}\label{prop:cantorimage}
Let $Y$ be a compact, doubling metric space. Then there is a compact metric space $X$ of Nagata dimension $0$ and a Lipschitz map from $X$ onto $Y$.
\end{proposition}
\begin{proof}
Let $Y$ be a compact, doubling metric space. Assume without loss of generality that $\diam(Y)=1$. 

For each $k\in\mathbb{N}$, let $N_k\subseteq Y$ be a sequence of nested $2^{-k}$-nets in $Y$, i.e., $N_1 \subseteq N_2 \subseteq \dots$. Given a point $y$ in some $N_k$, let 
$$N_{k+1}(y) := \{z\in N_{k+1} : d(y,z) \leq 2^{-k}\}.$$

Since $Y$ is bounded and doubling, there is an $M\in\mathbb{N}$ such that $|N_1|\leq M$ and $|N_{k+1}(y)| \leq M$ for each $k\in\mathbb{N}$ and $y\in N_k$.

We form $X$ as an abstract Cantor set, as follows. Let $X$ denote the set of infinite words in the alphabet $\mathcal{A}=\{1,2,\dots,M\}$, i.e.
$$ X = \{ (a_1, a_2, \dots) : a_i\in \mathcal{A} \text{ for each } i\in\mathbb{N}\}.$$
Define a metric $d$ on $X$ by
$$ d((a_1, a_2, \dots), (b_1, b_2, \dots)) = 2^{-\min\{i:a_i \neq b_i\}}.$$
It is standard that $d$ defines a compact metric (indeed, an ``ultra-metric'') on $X$.

Given $s>0$, we may choose $k$ with $2^{-(k+1)}\leq  s < 2^{-k}$. Given a word $w$ of length $k$ in the alphabet $\mathcal{A}$, the set $B_w$ of all elements of $X$ beginning with $w$ has diameter $2^{-(k+1)} \leq s$. Moreover, if $W\subseteq X$ has diameter $\leq s <  2^{-k}$, then $W$ is contained in some $B_w$ with $|w|=k$, from which it follows that the dijsoint cover 
$$\{B_w:|w|=k\}$$
of $X$ has $s$-multiplicity at most $1$. Thus, $X$ has Nagata dimension $\dim_N X = 0$.

We now define a Lipschitz map from $X$ onto $Y$. For each $k\in\mathbb{N}$ and $y\in N_k$, choose an arbitrary surjective map 
$$\phi_{k,y}:\mathcal{A} \rightarrow N_{k+1}(y),$$
which we can do since $|\mathcal{A}|=M\geq |N_{k+1}(y)|$.

Similarly, choose an arbitrary surjective map
$$ \phi_1: \mathcal{A}\rightarrow N_1.$$
Note that, for each sequence $(a_1, a_2, \dots)\in X$, the sequence
$$ y_1:=\phi_1(a_1), y_2:=\phi_{1,y_1}(a_2), y_3:=\phi_{2,y_2}(a_3)$$
is Cauchy in $Y$ as $d(y_i, y_{i+1})\leq 2^{-i}$. We therefore define a map $f\colon X\rightarrow Y$ by
$$ f((a_i)) = \lim_{i\rightarrow \infty} y_i,$$
where $y_i$ is defined as above.

We now show that $f$ is Lipschitz. Let $a=(a_i)$ and $b=(b_i)$ be distinct elements of $X$. Let $k\in \mathbb{N}\cup \{0\}$ be the length of the maximal shared initial segment between $a$ and $b$, so that $d(a,b) = 2^{-(k+1)}$. Then the first $k$ terms of the sequences $(y_i)$ defining $f(a)$ and $f(b)$ agree. Therefore,
$$ d(f(a), f(b)) \leq  \sum_{j=k}^\infty 2^{-j} = 2^{-(k-1)} = 4 d(a,b).$$
Hence, $f$ is Lipschitz.

Lastly, we show that $f$ is surjective. Note that, for each $k\in\mathbb{N}$ and $y\in N_k$, the point $y$ itself is an element of $\phi_{k,y}(\mathcal{A})$ because $N_k\subseteq N_{k+1}$. Thus, there is a choice of $(a_i)$ making all $y_i=y$ for all $i$ sufficiently large. This implies that $y=f((a_i))$, meaning that $N_k\subseteq f(X)$ for each $k\in\mathbb{N}$. It follows that $f(X)$ is dense in $Y$, and hence $f(X)=Y$ since $X$ is compact and $f$ is continuous.
\end{proof}

We immediately have the following corollary of Proposition \ref{prop:cantorimage} and Proposition \ref{prop:zerodim}.
\begin{corollary}\label{cor:0dimimage}
Every compact, doubling metric space is the image under a Lipschitz light map of a metric space with Lipschitz dimension $0$. 
\end{corollary}
In particular, using Corollary \ref{cor:carnotsubset}, we may find compact doubling spaces of infinite Lipschitz dimension, and thus by Corollary \ref{cor:0dimimage} see that Lipschitz light mappings may arbitarily increase the Lipschitz dimension of a metric space.

\subsection{Quasisymmetric and snowflake non-invariance}
As noted earlier, is easily apparent that Lipschitz dimension, like the Hausdorff, Assouad, and Nagata dimensions, is invariant under bi-Lipschitz deformations.

Nagata dimension is in addition a quasisymmetric invariant \cite[Theorem 1.2]{LS}. However, Lipschitz dimension is not.

\begin{corollary}\label{cor:nonsnow}
Lipschitz dimension is not a quasisymmetric, or even snowflake, invariant.
\end{corollary}
\begin{proof}
Let $\G$ be a non-abelian Carnot group. Every snowflake of $\G$ admits a bi-Lipschitz embedding into some $\RR^n$ by Assouad's embedding theorem \cite[Theorem 12.2]{He}. Therefore, each snowflake of $\G$ has finite Lipschitz dimension. On the other hand, $\dim_L \G = \infty$ by Theorem \ref{thm:Carnot}.
\end{proof}
More specifically, Corollary \ref{cor:nonsnow} shows that snowflake mappings can arbitrarily decrease Lipschitz dimension.

\begin{question}
Can a snowflake map increase the dimension of a compact metric space?
\end{question}

Recall that in Corollary \ref{cor:Rnsnowflake}, we showed that snowflakes of $\RR^n$ have Lipschitz dimension $n$. However, for general quasisymmetric deformations of Euclidean space, Lipschitz dimension is not an invariant. This follows from a construction of Semmes \cite{Semmes2}.

\begin{corollary}\label{cor:quasiRn}
There exists $n\in\mathbb{N}$ and an Ahlfors $n$-regular quasisymmetric deformation of $\RR^n$ with infinite Lipschitz dimension.
\end{corollary}
We recall that a metric space $X$ is \textit{Ahlfors $n$-regular} if there is a constant $C>0$ such that
\begin{equation}\label{eq:AR}
C^{-1} r^n \leq \HH^n(\overline{B}(x,r)) \leq Cr^n \text{ for all } r\leq \diam(X).
\end{equation}

\begin{proof}
By \cite[Theorem 1.15]{Semmes2}, there is an Ahlfors $n$-regular quasisymmetric deformation of $\RR^n$ that contains a bi-Lipschitz embedded copy of the Heisenberg group. Thus, it has infinite Lipschitz dimension by Theorem \ref{thm:Carnot}.
\end{proof}

The example provided by the proof of Corollary \ref{cor:quasiRn} must have $n> 4$, since $n$ arises in \cite[Theorem 1.15]{Semmes2} as the dimension of a Euclidean space containing a snowflake embedding of the Heisenberg group, which must have Hausdorff dimension greater than $4$. An interesting question would be to explore the Lipschitz dimension of quasisymmetric deformations of low-dimensional Euclidean spaces:

\begin{question}\label{q:quasiarc}
Does every quasi-arc (quasisymmetric image of $[0,1]\subseteq \RR$) have Lipschitz dimension $1$? Does every quasi-plane (quasisymmetric image of $\RR^2$) have finite Lipschitz dimension?
\end{question}
We note that a positive answer to Question \ref{q:nagata1} would imply a positive answer to the first part of Question \ref{q:quasiarc}.

\subsection{David--Semmes regularity and non-degenerate Lipschitz maps}\label{subsec:DSregular}
Another well-studied class of mappings are the so-called David--Semmes regular mappings.
\begin{definition}[\cite{DS97}, Definition 12.1]\label{def:DSregular}
A Lipschitz map $f\colon X\rightarrow Y$ is \textit{David--Semmes regular} if there is a constant $C>0$ such that, if $B=B(y,r)\subseteq Y$, then $f^{-1}(B)$ can be covered by at most $C$ balls of radius $Cr$ in $X$.
\end{definition}
David--Semmes regular mappings are finite-to-one in a controlled, quantitative manner.

Lipschitz light mappings need not be David--Semmes regular: David--Semmes regular mappings are always bounded-to-one, in particular discrete, whereas Lipschitz light mappings need not be. However, we do have the following direction.

\begin{lemma}\label{lem:DSLL}
David--Semmes regular mappings are Lipschitz light.
\end{lemma}
\begin{proof}
Let $f\colon X\rightarrow Y$ be David--Semmes regular. We may assume without loss of generality that the constant $C$ from Definition \ref{def:DSregular} is at least $1$.

Let $W$ be a set of diameter at most $r$ in $Y$. Then $f^{-1}(W)\subseteq X$ can be covered by a collection $\mathcal{B}$ of at most $C$ closed balls, each of radius $Cr$. 

Let $P=(x_1, \dots, x_k)$ be any $r$-path in $f^{-1}(W) \subseteq \cup_{B\in\mathcal{B}} B \subseteq X$. Without loss of generality, assume that $\diam(P) = d(x_1, x_k)$. Let 
$$ (B_1, B_2, \dots, B_m)$$
be a list of balls in $\mathcal{B}$ such that
\begin{equation}\label{eq:ballpath1}
x_1 \in B_1 \text{ and } x_k\in B_m,
\end{equation}
and
\begin{equation}\label{eq:ballpath2}
2B_i \cap 2B_{i+1} \neq \emptyset \text{ for each } i\in\{1,\dots, m-1\},
\end{equation}
and moreover such that $m$ is the minimal length of such a ``chain of balls'' satisfying \eqref{eq:ballpath1} and \eqref{eq:ballpath2}. Note that simply choosing one ball from $\mathcal{B}$ containing each $x_i$ yields a path satisfying \eqref{eq:ballpath1} and \eqref{eq:ballpath2}, so such chains exist.

The fact that $m$ is minimal implies that $B_i \neq B_j$ if $1\leq i < j\leq m$. Indeed, if $B_i = B_j$ for such $i,j$, then excising all the balls between indices $i+1$ and $j-1$ yields a shorter list satisfying  \eqref{eq:ballpath1} and \eqref{eq:ballpath2}.

Since there are only at most $C$ distinct balls in $\mathcal{B}$, we have $m\leq C$, and therefore
$$ \diam(P)= d(x_1, x_k) \leq 4C^2r.$$
As $P$ was an arbitary $r$-path in $f^{-1}(W)$, it follows that the $r$-components of $f^{-1}(W)$ have diameter at most $4C^2r$, which proves that $f$ is Lipschitz light.

\end{proof}

Combined with a result of David--Semmes, Lemma \ref{lem:DSLL} yields one way to control Lipschitz dimension of weak tangents of certain metric spaces by Hausdorff dimension. Recall the definition of Ahlfors regularity from \eqref{eq:AR}.
\begin{corollary}\label{cor:DS}
Let $X$ be an Ahlfors $n$-regular metric space, $Y$ a complete, doubling metric space, and $Z$ a compact subset of $X$. Suppose that there is a Lipschitz map $g:Z\rightarrow Y$ such that $\mathcal{H}^n(g(X))>0$. 

Then there are weak tangents $(\hat{X},\hat{x}) \in \WTan(X)$ and $(\hat{Y},\hat{y})\in\WTan(Y)$ such that $\dim_L\hat{X} \leq \dim_L \hat{Y}$.

In particular, if $\dim_L(Y)\leq m$, then $X$ has a weak tangent with Lipschitz dimension at most $m$.
\end{corollary}
\begin{proof}
By \cite[Proposition 12.8]{DS97}, there is a weak tangent
$$ \left((\hat{X},\hat{x}), (\hat{Y},\hat{y}),\hat{g}\right) \in \WTan(g)$$
such that $\hat{g}$ is David--Semmes regular. The map $\hat{g}$ is then Lipschitz light by Lemma \ref{lem:DSLL}, and result then follows from the observation at the beginning of subsection \ref{subsec:LLmapping}.

The ``In particular...'' statement follows from Corollary \ref{cor:Ldimtan}.
\end{proof}

Corollary \ref{cor:DS} can be viewed as a statement about which Ahlfors $n$-regular spaces can admit ``non-degenerate'' Lipschitz maps into other spaces, where a ``non-degenerate'' Lipschitz map is one whose image has positive $\HH^n$-measure. 

In particular, if $X$ is an Ahlfors $n$-regular metric space such that every weak tangent of $X$ has Lipschitz dimension greater than $m$, then $X$ cannot admit a non-degenerate Lipschitz map into a space of Lipschitz dimension at most $m$.

Specializing further, if $X$ is an Ahlfors $n$-regular metric space such that every weak tangent of $X$ has infinite Lipschitz dimension (as we've seen is the case for non-abelian Carnot groups), then $X$ cannot admit a non-degenerate Lipschitz map into any metric space of finite Lipschitz dimension. This is closely related to the notion of ``strong unrectifiability'' studied in \cite[Section 7]{AK} and \cite[Section 4.1]{GCDKin}, among other places.

\bibliography{lipdimbib}{}

\begin{thebibliography}{10}

\bibitem{AK}
L.~Ambrosio and B.~Kirchheim.
\newblock Rectifiable sets in metric and {B}anach spaces.
\newblock {\em Math. Ann.}, 318(3):527--555, 2000.

\bibitem{Bate}
D.~Bate.
\newblock Structure of measures in {L}ipschitz differentiability spaces.
\newblock {\em J. Amer. Math. Soc.}, 28(2):421--482, 2015.

\bibitem{BL}
D.~Bate and S.~Li.
\newblock Characterizations of rectifiable metric measure spaces.
\newblock {\em Ann. Sci. \'Ec. Norm. Sup\'er. (4)}, 50(1):1--37, 2017.

\bibitem{BS}
D.~Bate and G.~Speight.
\newblock Differentiability, porosity and doubling in metric measure spaces.
\newblock {\em Proc. Amer. Math. Soc.}, 141(3):971--985, 2013.

\bibitem{BJLPS}
S.~Bates, W.~B. Johnson, J.~Lindenstrauss, D.~Preiss, and G.~Schechtman.
\newblock Affine approximation of {L}ipschitz functions and nonlinear
  quotients.
\newblock {\em Geom. Funct. Anal.}, 9(6):1092--1127, 1999.

\bibitem{BP}
M.~Bourdon and H.~Pajot.
\newblock Poincar\'e inequalities and quasiconformal structure on the boundary
  of some hyperbolic buildings.
\newblock {\em Proc. Amer. Math. Soc.}, 127(8):2315--2324, 1999.

\bibitem{CDPT07}
L.~Capogna, D.~Danielli, S.~D. Pauls, and J.~T. Tyson.
\newblock {\em An introduction to the {H}eisenberg group and the
  sub-{R}iemannian isoperimetric problem}, volume 259 of {\em Progress in
  Mathematics}.
\newblock Birkh\"auser Verlag, Basel, 2007.

\bibitem{Ch99}
J.~Cheeger.
\newblock Differentiability of {L}ipschitz functions on metric measure spaces.
\newblock {\em Geom. Funct. Anal.}, 9(3):428--517, 1999.

\bibitem{CK13_inverse}
J.~Cheeger and B.~Kleiner.
\newblock Realization of metric spaces as inverse limits, and bilipschitz
  embedding in {$L_1$}.
\newblock {\em Geom. Funct. Anal.}, 23(1):96--133, 2013.

\bibitem{CK13_PI}
J.~Cheeger and B.~Kleiner.
\newblock Inverse limit spaces satisfying a {P}oincar\'e inequality.
\newblock {\em Anal. Geom. Metr. Spaces}, 3:15--39, 2015.

\bibitem{CKS}
J.~Cheeger, B.~Kleiner, and A.~Schioppa.
\newblock Infinitesimal {S}tructure of {D}ifferentiability {S}paces, and
  {M}etric {D}ifferentiation.
\newblock {\em Anal. Geom. Metr. Spaces}, 4:Art. 5, 2016.

\bibitem{ChenNiemeyer}
J.~P. Chen and R.~G. Niemeyer.
\newblock Periodic billiard orbits of self-similar {S}ierpi\'{n}ski carpets.
\newblock {\em J. Math. Anal. Appl.}, 416(2):969--994, 2014.

\bibitem{DS97}
G.~David and S.~Semmes.
\newblock {\em \normalfont ``{F}ractured fractals and broken dreams''},
  volume~7 of {\em Oxford Lecture Series in Mathematics and its Applications}.
\newblock The Clarendon Press, Oxford University Press, New York, 1997.

\bibitem{GCD15}
G.~C. David.
\newblock Tangents and rectifiability of {A}hlfors regular {L}ipschitz
  differentiability spaces.
\newblock {\em Geom. Funct. Anal.}, 25(2):553--579, 2015.

\bibitem{GCD16}
G.~C. David.
\newblock Bi-{L}ipschitz pieces between manifolds.
\newblock {\em Rev. Mat. Iberoam.}, 32(1):175--218, 2016.

\bibitem{GCDKin}
G.~C. David and K.~Kinneberg.
\newblock Lipschitz and bi-{L}ipschitz maps from {PI} spaces to {C}arnot
  groups.
\newblock {\em to appear, Indiana Univ. Math. J.}, {P}reprint, 2017.
\newblock ar{X}iv:1711.03533.

\bibitem{GCDK}
G.~C. David and B.~Kleiner.
\newblock Rectifiability of planes and {A}lberti representations.
\newblock {\em Ann. Sc. Norm. Super. Pisa Cl. Sci. (5)}, 19(2):723--756, 2019.
\newblock ar{X}iv:1611.05284.

\bibitem{DMR}
G.~De~Philippis, A.~Marchese, and F.~Rindler.
\newblock On a conjecture of {C}heeger.
\newblock In {\em Measure theory in non-smooth spaces}, Partial Differ. Equ.
  Meas. Theory, pages 145--155. De Gruyter Open, Warsaw, 2017.

\bibitem{DurandTyson}
E.~Durand-Cartagena and J.~T. Tyson.
\newblock Rectifiable curves in {S}ierpi\'{n}ski carpets.
\newblock {\em Indiana Univ. Math. J.}, 60(1):285--309, 2011.

\bibitem{Erdos}
P.~Erd\"{o}s.
\newblock The dimension of the rational points in {H}ilbert space.
\newblock {\em Ann. of Math. (2)}, 41:734--736, 1940.

\bibitem{EB}
S.~Eriksson-Bique.
\newblock Characterizing spaces satisfying {P}oincar\'{e} inequalities and
  applications to differentiability.
\newblock {\em Geom. Funct. Anal.}, 29(1):119--189, 2019.

\bibitem{GP}
N.~Gigli and E.~Pasqualetto.
\newblock Behaviour of the reference measure on {RCD} spaces under charts.
\newblock {P}reprint, 2016.
\newblock ar{X}iv:1607.05188.

\bibitem{He}
J.~Heinonen.
\newblock {\em \normalfont ``{L}ectures on analysis on metric spaces''}.
\newblock Universitext. Springer-Verlag, New York, 2001.

\bibitem{JLPS}
W.~B. Johnson, J.~Lindenstrauss, D.~Preiss, and G.~Schechtman.
\newblock Uniform quotient mappings of the plane.
\newblock {\em Michigan Math. J.}, 47(1):15--31, 2000.

\bibitem{Keith}
S.~Keith.
\newblock A differentiable structure for metric measure spaces.
\newblock {\em Adv. Math.}, 183(2):271--315, 2004.

\bibitem{KM}
M.~Kell and A.~Mondino.
\newblock On the volume measure of non-smooth spaces with {R}icci curvature
  bounded below.
\newblock {\em Ann. Sc. Norm. Super. Pisa Cl. Sci. (5)}, 18(2):593--610, 2018.

\bibitem{KL}
B.~Kleiner and B.~Leeb.
\newblock Rigidity of quasi-isometries for symmetric spaces and {E}uclidean
  buildings.
\newblock {\em Inst. Hautes \'{E}tudes Sci. Publ. Math.}, (86):115--197 (1998),
  1997.

\bibitem{KS}
B.~Kleiner and A.~Schioppa.
\newblock P{I} spaces with analytic dimension 1 and arbitrary topological
  dimension.
\newblock {\em Indiana Univ. Math. J.}, 66(2):495--546, 2017.

\bibitem{La00}
T.~Laakso.
\newblock Ahlfors ${Q}$-regular spaces with arbitrary ${Q}>1$ admitting weak
  {P}oincar\'e inequality.
\newblock {\em Geom. Funct. Anal.}, 10(1):111--123, 2000.

\bibitem{LPS}
U.~Lang, B.~Pavlovi\'{c}, and V.~Schroeder.
\newblock Extensions of {L}ipschitz maps into {H}adamard spaces.
\newblock {\em Geom. Funct. Anal.}, 10(6):1527--1553, 2000.

\bibitem{LS}
U.~Lang and T.~Schlichenmaier.
\newblock Nagata dimension, quasisymmetric embeddings, and {L}ipschitz
  extensions.
\newblock {\em Int. Math. Res. Not.}, (58):3625--3655, 2005.

\bibitem{LDtan}
E.~Le~Donne.
\newblock Metric spaces with unique tangents.
\newblock {\em Ann. Acad. Sci. Fenn. Math.}, 36(2):683--694, 2011.

\bibitem{LD_primer}
E.~Le~Donne.
\newblock A primer on {C}arnot groups: homogenous groups,
  {C}arnot-{C}arath\'{e}odory spaces, and regularity of their isometries.
\newblock {\em Anal. Geom. Metr. Spaces}, 5(1):116--137, 2017.

\bibitem{LDR}
E.~Le~Donne and T.~Rajala.
\newblock Assouad dimension, {N}agata dimension, and uniformly close metric
  tangents.
\newblock {\em Indiana Univ. Math. J.}, 64(1):21--54, 2015.

\bibitem{LeeSid}
J.~R. Lee and A.~Sidiropoulos.
\newblock Near-optimal distortion bounds for embedding doubling spaces into
  {$L_1$} [extended abstract].
\newblock In {\em S{TOC}'11---{P}roceedings of the 43rd {ACM} {S}ymposium on
  {T}heory of {C}omputing}, pages 765--772. ACM, New York, 2011.

\bibitem{Mo}
R.~Montgomery.
\newblock {\em A tour of subriemannian geometries, their geodesics and
  applications}, volume~91 of {\em Mathematical Surveys and Monographs}.
\newblock American Mathematical Society, Providence, RI, 2002.

\bibitem{Nagata}
J~Nagata.
\newblock {\em Modern dimension theory}, volume~2 of {\em Sigma Series in Pure
  Mathematics}.
\newblock Heldermann Verlag, Berlin, revised edition, 1983.

\bibitem{Pa89}
P.~Pansu.
\newblock M\'etriques de {C}arnot-{C}arath\'eodory et quasiisom\'etries des
  espaces sym\'etriques de rang un.
\newblock {\em Ann. Math., (2)}, 129(1):1--60, 1989.

\bibitem{Pauls}
S.~D. Pauls.
\newblock The large scale geometry of nilpotent {L}ie groups.
\newblock {\em Comm. Anal. Geom.}, 9(5):951--982, 2001.

\bibitem{Schioppa}
A.~Schioppa.
\newblock Derivations and {A}lberti representations.
\newblock {\em Adv. Math.}, 293:436--528, 2016.

\bibitem{Semmes2}
S.~Semmes.
\newblock On the nonexistence of bi-{L}ipschitz parameterizations and geometric
  problems about {$A_\infty$}-weights.
\newblock {\em Rev. Mat. Iberoamericana}, 12(2):337--410, 1996.

\end{thebibliography}
\bibliographystyle{plain}

\end{document}